\documentclass[a4paper,11pt]{amsart}
\usepackage{etex}

\usepackage{amscd}

\usepackage{amssymb}
\usepackage[all]{xy}
\usepackage{diagrams}
\usepackage{fullpage}
\usepackage{color}

\usepackage{euscript}
\usepackage{tikz}
\newlength{\defbaselineskip} 
\usepackage{footnote}
\usepackage{color}
\usepackage{todonotes}

\usepackage{wrapfig}

\usepackage[bookmarks, colorlinks, breaklinks, pdftitle={Verra fourfolds, twisted sheaves and the last involution},
pdfauthor={Chiara Camaere, Grzegorz Kapustka, Michal Kapustka, Giovanni Mongardi}]{hyperref}
\hypersetup{linkcolor=blue,citecolor=blue,filecolor=black,urlcolor=blue}

\newtheorem{thm}{Theorem}[section]
\newtheorem{cor}[thm]{Corollary}

\newtheorem{lemm}[thm]{Lemma}
\newtheorem{lem}[thm]{Lemma}
\newtheorem{prop}[thm]{Proposition}

\theoremstyle{definition}

\newtheorem{defn}[thm]{Definition}

\newtheorem{rem}[thm]{Remark}

\usetikzlibrary{matrix,arrows}
\makeatletter
\tikzset{
  edge node/.code={%
      \expandafter\def\expandafter\tikz@tonodes\expandafter{\tikz@tonodes #1}}}
\makeatother
\tikzset{
  subseteq/.style={
    draw=none,
    edge node={node [sloped, allow upside down, auto=false]{$\subseteq$}}},
  Subseteq/.style={
    draw=none,
    every to/.append style={
      edge node={node [sloped, allow upside down, auto=false]{$\subseteq$}}}
  }
}

 \numberwithin{equation}{section}
\numberwithin{equation}{section} \theoremstyle{definition}

\DeclareMathOperator{\Pic}{Pic}
\DeclareMathOperator{\Br}{Br}
\DeclareMathOperator{\Coh}{Coh}

\DeclareMathOperator{\divi}{div}
 
\DeclareMathOperator{\im}{im}

\DeclareMathOperator{\Mon}{Mon}
\DeclareMathOperator{\id}{id}
\DeclareMathOperator{\Aut}{Aut}
          \newcommand\PP{{\mathbb{P}}}
          \newcommand\IR{{\mathbb{R}}}
           
          \newcommand\C{{\mathbb{C}  }}

            \newcommand\ZZ{{\mathbb Z}}

          \newcommand\oo{\mathcal O}
             \newcommand\Q{\mathbb Q}
          \newcommand\Z{\mathbb{Z}}

          \newcommand\V{\mathcal{V}}
          \newcommand\rk{\mathrm{rk}}

\definecolor{zielony}{rgb}{0.5, 0.9, 0.1}
\definecolor{czerwony}{rgb}{0.8, 0.2, 0.1}
\definecolor{niebieski}{rgb}{0.3, 0.1, 0.9}

\newcommand{\hk}{hyper-K\"ahler }

\newcommand{\kahl}{K\"{a}hler }

\newcommand{\kntipo}{$K3^{[n]}$ type}
\newcommand{\kntiposp}{$K3^{[n]}$ type }

\newcounter{appendice}

\begin{document}
  
\title{Verra fourfolds, twisted sheaves and the last involution}
\author[C.~Camere]{Chiara Camere}
\address{Universit\`a degli Studi di Milano, Dipartimento di Matematica, Via Saldini 50, 20133 Milano, Italy}
\email{chiara.camere@unimi.it}
\author[G.~Kapustka]{Grzegorz Kapustka}
\address{Department of Mathematics and Informatics,
Jagiellonian University, {\L}ojasiewicza 6, 30-348 Krak{\'o}w, Poland. University of Z{\"u}rich}
\email{grzegorz.kapustka@uj.edu.pl}
\author[M.~Kapustka]{Micha\l{} Kapustka}
\address{University of  Stavanger, Department of Mathematics and Natural Sciences, NO-4036 Stavanger, Norway and Department of Mathematics and Informatics,
Jagiellonian University, {\L}ojasiewicza 6, 30-348 Krak{\'o}w, Poland}
\email{michal.kapustka@uis.no}
\author[G.~Mongardi]{Giovanni Mongardi}
\address{Alma Mater Studiorum Universit\`a di Bologna, Dipartimento di Matematica, P.zza di Porta San Donato 5, 40126, Bologna, Italy}
\email{giovanni.mongardi2@unibo.it}

\begin{abstract} We study the geometry of some moduli spaces of twisted sheaves on K3 surfaces.
In particular we introduce induced automorphisms from a K3 surface on moduli 
spaces of twisted sheaves on this K3 surface. As an application we prove the unirationality of moduli
spaces of irreducible holomorphic symplectic manifolds of $K3^{[2]}$-type 
admitting non symplectic involutions with invariant lattices
$U(2)\oplus D_4(-1)$ or $U(2)\oplus E_8(-2)$. This complements the 
results obtained in \cite{MW}, \cite{BCS}, and the results from \cite{IKKR} about the geometry of IHS fourfolds constructed using the Hilbert scheme of $(1,1)$ conics on Verra fourfolds. As a byproduct we find that 
IHS fourfolds of $K3^{[2]}$-type with Picard lattice $U(2)\oplus E_8(-2)$ naturally contain non-nodal Enriques surfaces.

\end{abstract}
\maketitle
\tableofcontents
\section{Introduction}

By an irreducible holomorphic symplectic (IHS) $2n$-fold (or hyper-K\"ahler manifold) we mean
a $2n$-dimensional simply connected compact K\"{a}hler manifold $X$ with trivial canonical
bundle that admits a unique (up to a constant)
closed non-degenerate holomorphic $2$-form $\omega_X$ 
(see \cite{Beauville}). In dimension two they are called $K3$ surfaces.

Let $G \subset  \Aut(X)$ be a cyclic group of automorphisms of prime order $p$ and fix a generator $\rho \in G$. If $\rho^{\ast}\omega_X = \omega_X $ then $G$ is called symplectic. Otherwise, there exists a primitive $p$-th root of unity $\zeta$ such that $\rho^{\ast}\omega_X = \zeta \omega_X$ and $G$ is called non- symplectic.
We denote by $H^2(X,\Z)^G=\{ \alpha\in H^2(X,\mathbb{Z})| \ g(\alpha)=\alpha \ \forall g\in G$\} the invariant lattice. 
General results on non symplectic involutions on IHS fourfolds were studied in 
\cite{B}.

We study IHS fourfolds deformation equivalent to Hilbert schemes of two points on $K3$ surfaces (i.e.~of $K3^{[2]}$ type).
In this case, the classification of invariant lattices of non symplectic automorphisms of prime order was done in \cite{BCS}, \cite{BCMS} and \cite{Tari}.
Next, explicit constructions of such automorphisms in all except four cases $U(2), U(2)\oplus D_4(-1), U(2)\oplus E_8(-2)$ and $p=23$ were done in \cite{O1}, \cite{BCS}, \cite{MW},  and for $p=5$ in \cite{Tari}.
In \cite{IKKR}, the authors constructed explicitly IHS fourfolds admitting an involution with invariant lattice $U(2)$. In this paper, we deal with the cases $U(2)\oplus D_4(-1), U(2)\oplus E_8(-2)$ leaving open the last problem of the explicit construction of IHS fourfolds of $K3^{[2]}$-type admitting an automorphism of order $23$ (cf. \cite{BCMS}). Notice that, though overlooked in \cite{MW}, the case $U(2)\oplus D_4(-1)$ can also be constructed as an induced involution on a moduli space of stable sheaves on a $K3$ surface; however in this paper we provide a different construction, of independent geometrical interest.
 The results concerning the classification of IHS fourfolds of $K3^{[2]}$ type admitting a non symplectic automorphism  
 can be summarized as follows:
 \begin{enumerate}
 \item Almost all cases of such fourfolds are moduli spaces of sheaves on a K3 surface; more precisely, from \cite{MW} we have the following:
  Let $X$ be a manifold of $K3^{[2]}$-type, and let $G \subset \Aut(X)$ be a group of non symplectic automorphisms. Assume that the induced action of $G$ on the Mukai lattice containing $H^2(X,\mathbb{Z})$ fixes a copy of $U$, then there exists a $K3$ surface $S$ such that $G \subset \Aut(S),$ $X$ is a moduli space of stable objects on $S$ and the action of $G$ on $X$ is induced by that on $S$.
 This case covers all the possibilities except the cases below, and as we mentioned also includes the case $H^2(X,\Z)^G\cong U(2)\oplus D_4(-1)$.
\item The automorphism has order two and the IHS fourfold is a degeneration of a double EPW sextic (see \cite{Ferretti}, \cite[\S 5.2]{MW} and \cite{O1}). Here we have two kinds of degenerations, those which can be obtained also as in the previous case and those that cannot. The latter have invariant lattice with maximal discriminant group, and the fact that these involutions cover all cases claimed in \cite[\S 5.2]{MW} is sketched in Subsection \ref{moduli}. 
\item The automorphism has order three and the invariant lattice is either $ \langle 6 \rangle$ or $ \langle 6 \rangle \oplus E_6^{\vee}(-3)$. Both lattices are realized as invariant lattices of automorphisms on Fano varieties of lines on a cubic fourfold (see \cite[\S 6.2]{BCS}), and the fact that all families are given by the above said Fano varieties will be proven in \cite{BCScoming} in the case $ \langle 6 \rangle$ and is still open for $ \langle 6 \rangle \oplus E_6^{\vee}(-3)$.
\item The automorphism of $X$ has order two and the invariant lattice is $U(2)$. Then, from \cite{IKKR} we infer that
$X$ is a double cover of a special quartic section of the cone over $\PP^2\times \PP^2$ (described in Section \ref{EPW quartic}). 
\item The automorphism has order two and the invariant lattice is $U(2)\oplus E_8(-2)$.
\item  $p=23$, $X$ is unique with Picard number one and admits a polarization of degree $46$ (see \cite{BCMS}).
\end{enumerate}
The main result of this paper is a description of IHS fourfolds in Case $(5)$ above. 
We shall see that in this case IHS fourfolds are isomorphic to moduli spaces of twisted sheaves and the
automorphisms are induced from the underlying K3 surface. In Section \ref{induced}, we introduce and study induced automorphisms on the moduli spaces of twisted sheaves.

Next, we study the geometry of the examples in case (5) and of their moduli spaces.
A Verra threefold is a bidegree $(2,2)$ divisor in $\PP^2\times \PP^2$.
We call a Verra fourfold a $2:1$ cover of $\PP^2\times \PP^2$ branched along a Verra threefold.
Such a fourfold $V$ admits two quadric fibrations over $\PP^2$ and a natural embedding $V\subset \PP^9$ (see Section \ref{EPW quartic}).
We shall consider conics on $V\subset \PP^9$ that map to lines through the projections: we call them $(1,1)$ conics.
Recall that such a $(1,1)$ conic $c$ is contained in a natural del Pezzo surface $D_c$ of degree 4 (see Section \ref{EPW quartic}) and moves in a pencil on this surface.
The latter induces a $\mathbb{P}^1 $ fibration on the Hilbert scheme of $(1,1)$ conics. It was proven in \cite{IKKR} that the base of this fibration is a fourfold of $K3^{[2]}$-type that we shall denote by $X_V$.
Consider now the involution $\iota$ on $\PP^2\times \PP^2$ induced by non trivial involutions on both factors. 

Our main result is the following:
\begin{thm} \label{main} 
\begin{enumerate}
\item If $(X,i)$ is a general IHS fourfold of $K3^{[2]}$-type with non symplectic involution having invariant lattice $E_8(-2)\oplus U(2)$, then there exists a Verra fourfold $V\subset \PP^9$ constructed from a Verra threefold symmetric with respect to the involution $\iota$ such that $X$ is isomorphic to $X_V$.

\item If $(X,i)$ is a general IHS fourfold of $K3^{[2]}$-type with non symplectic  involution having invariant lattice $D_4(-1)\oplus U(2)$ then there exists a Verra fourfold $V\subset \PP^9$ constructed from a Verra threefold $V'\subset \PP^8$ with singularity
of analytic type $x^2 + y^2 + z^3 + t^3 = 0$ in $\C^4$. 
\end{enumerate}
\end{thm}
In Section \ref{moduli} we clarify what we mean by general in the above theorem by considering appropriate moduli spaces of IHS fourfolds with involutions. 
Note that a general IHS fourfold with non symplectic involution having invariant lattice $E_8(-2)\oplus U(2)$ admits a symplectic involution (and two non symplectic ones).
It follows from Section \ref{induced} that the considered IHS fourfolds are moduli spaces of twisted sheaves such that the non-symplectic involution is induced. We are however interested in a geometric construction.
It is natural to consider in Section \ref{Verras} two types of symmetric Verra fourfolds $\V_1$ and $\V_2$. 
Each Verra fourfold admits two quadric fibrations over $\PP^2$ with sextic curves as discriminant. 
One of these sextics is smooth and we denote it by $C_1$ (for a general example).
 It follows from Section \ref{derived} that those hyper-K\"ahler fourfolds are
birational to the moduli space of twisted sheaves on the K3 surface being the 
double cover of $\PP^2$ branched along $C_1$.
We prove in Section \ref{symmetric two torsion} that the general symmetric sextic occurs as the discriminant of a symmetric Verra threefold.
 This allows us to find the Picard group of hyper-K\"ahler fourfolds constructed from $\V_1$ and $\V_2$
 that are either $U(2)\oplus E_8(-2)$ or  $U\oplus E_8(-2)$.
 We conclude by an analysis of the fixed loci (see Section \ref{fixed}) of the considered involutions showing  that the examples constructed from $\mathcal{V}_1$ and $\mathcal{V}_2$ are different.

In the case $U(2)\oplus D_4(-1)$, the corresponding IHS fourfolds can be constructed by considering singular Verra fourfolds with singularities as in Theorem \ref{main}. This family of non symplectic automorphisms can be also constructed as induced automorphisms on the moduli space of sheaves on a K3 surface being the minimal resolution of the double cover of $\PP^2$ branched along a sextic with four nodes.

 Recall that involutions on K3 surfaces were studied for example in 
 \cite{Mo,vGS,V,AST,N2,OZ}.
The unirationality of the moduli space of K3 surfaces with a symplectic involution 
was studied in \cite{V} and non symplectic in \cite{Ma, DK}.
The unirationality of moduli spaces of IHS fourfolds of $K3^{[2]}$ type with involutions from case (1) can be often considered by using the  unirationality of the moduli spaces of the associated K3 surfaces (see Proposition \ref{moduli1}).
The situation in the case (5) is more complicated since there are no direct bijections between the moduli space of such IHS manifolds with a moduli space of K3 surfaces. 
Using our descriptions  of families of IHS fourfolds with Picard group from case $(5)$, we obtain as a consequence the following.

\begin{cor}\label{cor} There exist flat families with unirational bases parametrizing polarised IHS fourfolds of $K3^{[2]}$-type with polarisations that admit non symplectic involutions with invariant lattice either $U(2)\oplus E_8(-2)$ or $U(2)\oplus D_4(-1)$.
\end{cor}
In fact, a general element of the family corresponding to the lattice $U(2)\oplus E_8(-2)$ admits three involutions one symplectic and two non symplectic with invariant lattices $U(2)$ and $U(2)\oplus E_8(-2)$. In this case, we construct a flat family of polarized IHS with polarization inducing the involution with invariant lattice $U(2)$. After composing this involution with a symplectic one we infer the one with fixed lattice $U(2)\oplus E_8(-2)$. We discuss the moduli spaces related to our construction in Section \ref{moduli}.

We first introduce some lattice theory needed in the following, then we recall the constructions from \cite{IKKR} of IHS fourfolds with involutions with invariant lattice $U(2)$; next, we introduce moduli spaces of twisted sheaves. 
In Section \ref{fixed}, we present a relation between the geometry of 
IHS fourfolds of $K3^{[2]}$ type with Picard lattice $E_8(-2)\oplus U(2)$ and some classical constructions.
We prove in Proposition \ref{fixed pts prop} that one of the non symplectic involutions has a general (non nodal) Enriques surface
as fixed locus. One can also show that the family of IHS fourfolds with Picard lattice $E_8(-2)\oplus U(2)$ is a mirror family (in the sense of \cite{Cam-AIF}) to the family of Hilbert squares $S^{[2]}$ such that  $S$ is a $K3$ cover of an Enriques surface. Note that the geometry of the fourfold $S^{[2]}$ was studied in \cite{DM} in relation to the Morin problem, it would be interesting to relate these results to our constructions.

\subsection*{Acknowledgements}
GK was supported by the project 2013/08/A/ST1/00312, MK, GM by the project
2013/10/E/ST1/00688.  
We would like to thank A. Iliev, A. Kresch, A. Kuznetsov, C. Meachan, Ch. Okonek, J. Ottem, K. Ranestad, 
P. Stellari for discussions. Finally, we want to express our gratitude to the anonymous referee for his/her precious remarks on the first version of this paper.
\section{Preliminary notions}
\subsection{Lattices}\label{ssec:lattices}
In this subsection, we briefly sketch some properties of lattices, more complete references are \cite{Nikulin_2el} and \cite{N2}. A lattice $L$ is a free $\mathbb{Z}$ module endowed with a symmetric integral form $(\cdot,\cdot)$, usually non degenerate. It is called even if the associated quadratic form takes only even values and unimodular if $L$ is isometric to the dual lattice $L^\vee$. To each even lattice, we can associate two fundamental invariants called the discriminant group and discriminant form. The group is given by the quotient $A_L:=L^\vee/L$ and the discriminant form $q_{A_L}$ is the bilinear form induced on $L^\vee$, with values modulo $2\mathbb{Z}$ (for the quadratic form) and modulo $\mathbb{Z}$ (for the bilinear form). The signature of a lattice is the signature of the bilinear form on $L\otimes \mathbb{R}$. 
We say that two even lattices belong to the same genus if they have the same rank, signature and discriminant group.
In the rest of the paper, we will deal mainly with the following lattices:
\begin{itemize}
\item Root lattices, generated by a set of roots (vectors of square $2$), with intersection of the roots determined by the corresponding Dynkyn diagram. The cases of interest are $D_4$ and $E_8$ (which is the unique unimodular rank eight positive definite lattice).
\item The hyperbolic plane $U$, generated by two elements of square $0$ and self intersection $1$.
\item A lattice $L$ with its quadratic form multiplied by $n$, which is denoted $L(n)$.
\item Orthogonal direct sums of the previous ones, denoted with the symbol $\oplus$.
\end{itemize}
A lattice $L$ is called $2$-modular (or $2$-elementary) if its discriminant group $A_L$ is $(\mathbb{Z}/2\mathbb{Z})^a$ for some $a$. An important invariant of $2$-modular lattices, usually called $\delta$, takes value $0$ if the discriminant form has values in $\mathbb{Z}/2\mathbb{Z}$ and $1$ otherwise. Nikulin \cite[Thm. 4.3.1 and 4.3.2]{Nikulin_2el} proved the following:

\begin{thm}\label{thm:nik_2el}
An even $2$-modular lattice of rank $r$ and signature $(p,r-p)$ is uniquely determined by the invariants $(r,p,a,\delta)$.
\end{thm}

A sublattice $N\subset L$ is primitive if $L/N$ is torsion-free. In the opposite case, if $L/N$ is torsion, $L$ is said to be an overlattice of $N$, and, in this case, $\rk N=\rk L$ and they have the same signature; isometry classes of overlattices $L$ of a given lattice $Z$ are determined by isomorphism classes of isotropic subgroups $H\subset A_Z$. In particular, if $T$ is a primitive sublattice of $L$ and $Z$ is its orthogonal, $L$ is an overlattice of $T\oplus Z$, the finite group $H=L/(T\oplus Z)$ is isotropic inside $A_T\oplus A_Z$ and the following equality holds:
\[[L:(T\oplus Z)]^2=\frac{|A_T||A_Z|}{|A_L|}.\]
Moreover, we have that $H^\perp/H=A_L$.
For an element $l\in L$, we can define the divisibility $\divi(l):=d$ as the positive generator of the ideal $(l,L)$. This also defines a map, which we will call discriminant projection, $L\rightarrow A_L$ where the image of $l\in L$ is the class of $l/\divi(l)$. 

For our purposes the following criterion, called usually Eichler's criterion, will be useful:
\begin{thm}\cite[Thm. 2.9]{BHPV}\label{thm:eichstrong}
Let $L,M$ be two even lattices and let $U^r\subset M$ and $A_L\subset A_M$ with the additional condition $q_{A_L}=q_{A_M|A_L}$. Then the following holds:
\item If $r\geq\rk(L)$, then there exists a primitive embedding of $L$ in $M$
\item If $r\geq\rk(L)+1$ then all primitive embeddings of $L$ in $M$ with isometric discriminant projections of $L$ in $M$ are equivalent by an isometry of $M$.
\end{thm}
We will be interested also in some lattices associated with $K3$ or abelian surfaces: recall that to such a surface $S$ one can associate a weight two Hodge structure given by $H^{2*}(S)$, where the $2,0$ part coincides with $H^{2,0}(S)$, whose integer part is called the Mukai lattice $\Lambda$ (or $\Lambda_S$ if needed) and is endowed with the intersection pairing. An element $(a,b,c)\in H^{2*}(S,\mathbb{Z})=\Lambda$, with $a\in H^0$, $b\in H^2$ and $c\in H^4$, is called a Mukai vector.

\subsection{Double EPW quartic}\label{EPW quartic}
Let us introduce the notation from \cite{IKKR} and \cite{IKKR1} necessary to study the geometry of the IHS fourfolds considered in this paper. 
We recall the construction of a $19$-dimensional family of IHS fourfold of $K3^{[2]}$-type with non symplectic involution and Beauville-Bogomolov degree $q=4$.
These IHS fourfolds are constructed as natural double covers of Lagrangian degeneracy loci in the cone over $\PP^2\times \PP^2$. Let us be more precise. 

Let $U_1$ and $U_2$ be two three dimensional complex vector spaces, fix a volume form on each space $U_1$, $U_2$ such that
$\wedge^2U_1 \simeq U_1^{\vee}$ and $\wedge^2U_2 \simeq U_2^{\vee}$ and let $$\eta\colon \wedge^3U_1 \otimes \wedge^3U_2\to \C$$ be the product volume form. 
Let $V\subset \PP^9$ be the intersection of the cone $$C(\PP(U_1 )\times \PP(\wedge^2 U_2 )) \subset \PP(\wedge^3 U_2 \oplus(U_1 \otimes \wedge^2 U_2 ))$$
with a quadric hypersurface. Note that $V$ can be seen as the double cover of $\PP^2\times \PP^2$ branched along a $(2,2)$ divisor and we call it a \emph{Verra fourfold}. The two projections $\pi_1$ and $\pi_2$ to $\PP^2$ define quadric bundle structures on $V$ with sextic curves as discriminant loci, cf. \cite{Vprym}.
Denote by $F(V)$ the Hilbert scheme of conics on $V\subset \PP^9$ projecting to lines through $\pi_1$ and $\pi_2$, i.e~$(1,1)$ conics.

Let $c\in F(V)$, then $c\subset \PP^9$ spans a plane $P_c \subset \PP(\wedge^3 U_2 \oplus
(U_1 \otimes \wedge^2U_2 ))$. Consider the locus $H$ of quadrics containing $V \cup P_c\subset \PP^9$. Now if we assume that $V$ is smooth then for each $c$  the space of quadrics containing $V\subset \PP^9$ 
is a hyperplane (see \cite[Lem.~3.1]{IKKR}) that we denote by $H_c$, i.e.~naturally a point $H_c \in \PP(\wedge^3U_2 \oplus (\wedge^2U_1\otimes U_2 ))$. We defined a morphism
$$\varphi_V \colon F(V)\to \PP(\wedge^3U_1 \oplus (\wedge^2U_1 \otimes U_2 )). $$
The image $Y_V$ of $\varphi_V$ is isomorphic to the intersection of the cone over the Segre embedding $C(\PP^2\times \PP^2)\subset \PP^9$ with a special quartic hypersurface (explicitly described in \cite{IKKR}). We call it an \emph{EPW quartic section}.
The map $\varphi\colon F(V)\to Y_V$ factorizes by a $\PP^1$ fibration $F(V)\to X_V$ and a $2:1$ cover $X_V\to Y_V$. For general $V$ ($V$ smooth is not enough) the manifold $X_V$ is an IHS fourfold of $K3^{[2]}$-type with Beauville-Bogomolov degree $4$. 

Recall that $Y_V\subset C(\PP(\wedge^2 U_1)\times \PP(U_2))$ can be described (see the proof of \cite[Prop.~3.2]{IKKR}) as a Lagrangian degeneracy locus. 
Let us denote $W=U_1\oplus U_2$, then for  $A\in LG_{\eta}(10,\wedge^3(W))$, the authors considered in \cite{IKKR1}
$$D^A_i=\{v\in G(3,W) | \ \dim (\PP(T_v)\cap \PP(A))\geq i  \},$$
with $\PP(T_v)$ the embedded tangent space at $v$ to $G(3,W)$,
and called $D_2^A$ an EPW cube.
The cone $C(\PP(\wedge^2 U_1)\times \PP(U_2))$ can be seen as the intersection of the Grassmannian $G(3,W)$ with $\PP(T_{[U_1]})$, its embedded tangent space at $[U_1]$:
$$C(\PP(\wedge^2 U_1)\times \PP(U_2))=\PP(W\wedge (\wedge^2U_1 ))\cap G(3, W)\subset \PP(\wedge^3 W).$$
For $Y_V\subset C(\PP^2\times \PP^2)$, we can find $A\in LG_{\eta}(10,\wedge^2 W)$ with $[U_1]\in G(3,W)\cap \PP(A)$ 
such that $$Y_V=D_2^A\cap \PP(W\wedge (\wedge^2U_1 ))=Y_A.$$

Finally one can reconstruct from  such $A$ the Verra threefold $V_A$ that gives $Y_A$ through the map $\varphi_{V}$.
Observe that for any choice of disjoint spaces $U_1, U_2$ the form $\eta$ provides an isomorphism:
$$T_{[U_1]}/<[U_1]>=\wedge^2 U_1\otimes U_2\simeq   (\wedge^2 U_2\otimes U_1)^{\vee}= (T_{[U_2]}/<[U_2]>)^{\vee}
$$
For our choice of $[U_2]\in G(3,W)$, consider now
$$q_{A,U_2}: T_{[U_2]}/<[U_2]> \to T_{[U_1]}/<[U_1]>\simeq (T_{[U_2]}/<[U_2]>)^{\vee}$$ the symmetric map whose graph is $\bar{A}:=A/<[U_1]>$ and let $Q_{A,U_2}$ be the corresponding quadric. 
Let $C_{U_2}=\PP(T_{[U_2]})\cap G(3,W)$ and $P_{U_2}=\PP(\wedge^2 U_2)\times \PP(U_1)$ be the corresponding Segre embedding $$\mathbb{P}^2\times \mathbb{P}^2\subset \PP(T_{[U_2]}/<[U_2]> )\simeq \PP( (T_{[U_1]}/<[U_1]>)^{\vee}).$$
Define $V'_{A,U_2}=P_{U_2}\cap Q_{A,U_2}$ the Verra threefold associated to $A$ and $U_2$ (and the volume forms on $U_1$ and $U_2$).

\begin{lem}[cf.~\cite{IKKR}] The threefold
$V'_{A,U_2}$ is independent of $U_2$ and of the volume forms. We denote it by $V'_A$.
\end{lem}
\begin{proof}
 Observe that, if we choose a different  $[U'_2]\in G(3,U)$, we have a canonical isomorphism induced by the symplectic form $T_{[U'_2]}/<[U'_2]> \simeq(T_{[U_1]}/<[U_1]>)^{\vee}\simeq  T_{[U_2]}/<[U_2]>$ and under this identification we have $Q_{A,U_2}-Q_{A,U'_2}\in I_{P_{U_2}}\simeq I_{P_{U_2'}}$
\end{proof}

Finally let $V'\in \PP(U_1)\times \PP(\wedge^2 U_2)$ be a  Verra threefold. We claim that $V'$ appears in the construction above. Indeed $V'$ is obtained as the intersection 
 $\PP(U_1)\times \PP(\wedge^2 U_2)$ with some quadric hypersurface $Q\subset \mathbb{P}(U_1\otimes \wedge^2 U_2)$. Then $Q$ is defined by a quadric equation which corresponds to a symmetric linear map:
$$\tilde{q}: U_1\otimes \wedge^2 U_2 \to (U_1\otimes \wedge^2 U_2)^{\vee}$$ 
If we now fix  a volume form on $W=U_1\oplus U_2$ we obtain an identification:
$$(U_1\otimes \wedge^2 U_2)^{\vee}=  \wedge^2 U_1\otimes U_2$$
which composed with $\tilde{q}$ gives a linear map:
$$q: U_1\otimes \wedge^2 U_2 \to \wedge^2 U_1\otimes U_2$$ 
whose graph is a space $\bar{A}\subset U_1\otimes \wedge^2 U_2 \oplus \wedge^2 U_1\otimes U_2$ which is Lagrangian with respect to the 2-form $\eta'$ given by wedge product in $U_1\otimes \wedge^2 U_2 \oplus \wedge^2 U_1\otimes U_2$. Now $$A=\bar{A}\oplus \wedge^3 U_1 \subset \wedge^3 (U_1\oplus U_2)$$ is a Lagrangian space with respect to the form $\eta$ on $\wedge^3 (U_1\oplus U_2)$. Then $V'_{A,U_2}$ is isomorphic to $V'$.

\subsection{Twisted sheaves on $K3$ or abelian surfaces}\label{ssec:twist}
In this subsection, we provide a quick introduction to twisted sheaves, focused on the few concepts needed for our purposes. For a concise yet rather complete overview of the topic, we refer to \cite{H2}. Let $S$ be a $K3$ or abelian surface and let $\alpha$ be a class of finite order in the Brauer group $\Br(S):=H^2(\mathcal{O}^*_S)$. The category $\Coh(S,\alpha)$ represents coherent sheaves whose second gluing condition is twisted by a \v{C}ech cocycle of class $\alpha$. Thus, we define the twisted $K3$ surface $(S,\alpha)$ by defining its category of twisted sheaves. Two twisted K3 surfaces are isomorphic if there is an isomorphism of the two surfaces sending one Brauer class into the other.

Recall, from \cite{vG}, that an element from the Brauer group can be seen as a  homomorphism
from the transcendental lattice $T(S)\to \Q/\ZZ$. Moreover, an element $\alpha$ of order $n$ defines a surjective homomorphism
$T(S)\to \ZZ/n\ZZ$, whose kernel is an index $n$ sublattice $\Gamma_\alpha$ of the transcendental lattice $T(S)$. By the exponential sequence, we can lift the class $\alpha$ to an element in $H^2(S,\mathbb{Q})$, which is determined up to elements in $H^2(S,\mathbb{Z})+\frac{1}{n}\Pic(S)$. Such an element $B$ is called a $B$-lift of $\alpha$. 
The $B$-lift allows us to define a Hodge structure $\widetilde{H}(S,\alpha)$ for the $K3$ surface $S$, as in \cite[Lem.~3.1]{MS}: the symplectic form is defined as $\sigma_{S,B}:=\sigma_S+B\wedge \sigma_S$, where $\sigma_S$ is the symplectic form on $S$ and the definition only depends on the $(0,2)$-part of $B$. This twisted Hodge structure lies inside the lattice $H^{2*}(S)$ and, as a lattice, is a finite order extension of the lattice generated by $H^2(S,\ZZ)$, the class $(0,0,1)$ and the class $(k,kB,kB^2/2)$, where $k$ is the smallest integer such that $(k,kB,kB^2/2)$ is an integer class (\cite[Proof of Lemma 2.5]{H}). The Brauer classes which we will consider in this paper admit a lift with $B^2=0$ and $B.H=0$; in this case $\Pic(S,\alpha)$ is the sublattice generated by $\Pic(S),$ $(0,0,1)$ and the class $(n,nB,0)$, where $n$ is the order of $\alpha$, because this sublattice is primitive. Different $B$-lifts produce isometric Hodge structures.

\begin{rem}
Let $(S,\alpha)$ be as above, and let $\varphi\in \Aut(S)$ be of finite order such that $\varphi(\alpha)=\alpha$, and let $\varphi(\sigma)=\xi\sigma$ for some root of unity $\xi$ (possibly trivial). Let us look at the induced action of $\varphi$ on $\widetilde{H}(S,\alpha)$. Let $B$ be a $B$-lift of $\alpha$, notice that also $\varphi(B)$ is a $B$-lift of $\alpha$. Then $\varphi(\sigma_B)=\varphi(\sigma)+\varphi(B)\wedge \varphi(\sigma)=\xi\sigma_{\varphi(B)}$, and $\varphi$ acts trivially on $\langle (0,0,1),(k,kB,kB^2/2) \rangle$. (Actually, it sends $(k,kB,kB^2/2)$ to $(k,k\varphi(B),k\varphi(B)^2/2)$ but these two classes are the same in $\widetilde{H}(S,\alpha)$).
\end{rem}

It is possible to construct moduli spaces of twisted sheaves as in \cite{Yo}, which naturally have a symplectic form (indeed, they are of $K3^{[n]}$ type if $S$ is a $K3$, and a fibre of their Albanese map is of Kummer type if $S$ is abelian). Given an ordinary Mukai vector $v=(r,H,s)\in H^{2*}(S)$, one can define a twisted Mukai vector $v_{B}$ as $v_B:=(r,H+rB,s+B\wedge H+rB^2/2)$, which is an element of $\widetilde{H}(S,\alpha)$ if and only if $n$ divides $r$. There are several definition of the twisted Mukai vector, for a complete comparison we refer to \cite[Section 4]{PT}, here we have used Huybrechts--Stellari definitions. Then a moduli space of twisted sheaves with Mukai vector $v_B$ exists, has dimension $v_B^2+2$ and is denoted $M_{v_B}(S,\alpha)$. Moreover, for $v_B^2\geq 2$ (or 6 if $S$ is abelian) there is the following isometry:
$$v_B^\perp \cong H^2(M_{v_B}(S,\alpha),\mathbb{Z}),$$
where the orthogonal is taken inside the Mukai lattice. 
For our purposes, it is more convenient to denote the moduli space of twisted sheaves with an untwisted Mukai vector and a Brauer class, as different choices of a $B$-lift determine equivalent twisted Mukai vectors $v_B$ (and the same category $\mathcal{C}oh(S,\alpha)$).
\begin{rem}
The paper \cite{Yo} focuses mainly on moduli spaces of twisted sheaves on $K3$ surfaces, the above mentioned results for sheaves on abelian surfaces can be found in related works of Yoshioka \cite{yoshi_der}, and Minamide, Yanagide and Yoshioka \cite[Section 4]{myy}.
\end{rem}

\section{Induced automorphisms on moduli spaces of twisted sheaves}\label{induced}
In this section we will generalize the work of the last author and Wandel \cite[Thm. 4.4 and 4.5]{MW} to the case of moduli spaces of twisted sheaves and their automorphisms induced by automorphisms of the underlying surface. This construction will provide a first way to construct the family of fourfolds with the desired non-symplectic involutions; moreover, we obtain as a side result Theorem \ref{thm:twistinduced}.

 The main technical tool here is a result of Huybrechts \cite[Lemma 2.6]{H}, which extends to the twisted case a result of Addington \cite[Prop. 4]{A} and the above cited work \cite[Prop. 3.3 and 3.4]{MW}. The original result is phrased only for manifolds which are of \kntipo, but it easily extends to generalized Kummer's:
Let $X$ be a manifold of \kntiposp (resp. Kummer type) and let $$i\,:\,H^2(X,\mathbb{Z})\rightarrow \Lambda_{K3}:=U^4\oplus E_8(-1)^2$$ (resp. $\Lambda_{Ab}:=U^4$) be a primitive embedding. Endow $\Lambda_{K3}\otimes \C$ (resp. $\Lambda_{Ab}\otimes \C)$ with the weight two Hodge structure given by the symplectic form on $X$, that is $\Lambda_{K3}^{2,0}=i(\sigma_X)$ (resp. $\Lambda_{Ab}^{2,0}=i(\sigma_X)$). 
\begin{prop}\cite{H}
Keep notation as above, then $X$ is a moduli space of twisted sheaves (resp. the Albanese fibre of a moduli space of twisted sheaves) with a $n$ torsion Brauer class on a $K3$ surface (resp. abelian surface) if and only if $\Lambda_{K3}^{1,1}$ (resp. $\Lambda_{Ab}^{1,1}$) contains a copy of $U(k)$, for some multiple $k$ of $n$.
\end{prop}
We do not give a proof of this result, let us just sketch the main ingredients. On one hand, if we start with a $K3$  (resp. abelian) surface $S$ and take a $n$ torsion Brauer class $\beta$, as in the previous section we have that $\Pic(S,\beta)$ contains primitively $\Pic(S)$ and contains $\langle (0,0,1),(k,kB,kB^2/2)\rangle\cong U(k)$ for a choice of $B$-lift $B$ of $\beta$. When we saturate this lattice, we obtain the algebraic part of the Mukai lattice of any moduli space of twisted sheaves on $(S,\beta)$, which therefore contains $U(k)$.
  On the other hand, if we start with $X$ and $\Lambda_{K3}^{1,1}$ (resp. $\Lambda_{Ab}^{1,1}$) and we call $P$ the orthogonal to the copy of $U(k)$ which we found, there exists a $K3$ surface $S$ and a Brauer class $\beta$ on it such that $P\cong \Pic(S)$ and $\Pic(S,\beta)\cong \Lambda_{K3}^{1,1}$ (resp. $\Lambda_{Ab}^{1,1}$) is the twisted Picard lattice of $(S,\beta)$. It follows that $X$ has the period of a moduli space of twisted sheaves on $S$. But now the result follows by the global Torelli, because all manifolds Hodge isometric to moduli spaces of twisted sheaves are of the same form.\\
If $S$ is a surface, $G$ is a group of automorphisms on it, $v_B$ is a $G$-invariant twisted Mukai vector and $\beta$ is a $G$-invariant Brauer class, the moduli space $M_{v_B}(S,\beta)$ has a natural $G$-action which we call induced (or twisted induced, if the need arises). Indeed, we have the following
\begin{lem}
Let $S$ be a symplectic surface and $G\subset \Aut(S)$ be any group. Let $\mathcal{M}$ be a component of a moduli space of (twisted) sheaves which is fixed by $G$. Then $G\subset \Aut(\mathcal{M})$ unless $\mathcal{M}$ is a point.
\end{lem}
\begin{proof}
Unless $\mathcal{M}$ is a point (or $\mathcal{M}$ fibres to a point with the Albanese map), there is an embedding $\widetilde{H}(S,\beta)\rightarrow \Lambda(\mathcal{M})$, where $\beta$ is a Brauer class ($G$-invariant by hypothesis) and $\Lambda(\mathcal{M})$ denotes the Mukai lattice attached to $\mathcal{M}$. This embedding is $G$-equivariant, as proven in \cite[Lemma 2.34]{MW}, therefore the action of $G$ on $\mathcal{M}$ is faithful.
\end{proof}

 If an automorphism group looks like an induced one, we will call it numerically induced. More precisely we have the following:
\begin{defn}
Let $X$ be a manifold of \kntiposp or of Kummer type. Let $G$ be a group of automorphisms of $X$ such that the following are satisfied:
\begin{itemize}
\item The group $G$ is finite.
\item The action of $G$ induced on the discriminant group $H^2(X,\mathbb{Z})^\vee/H^2(X,\mathbb{Z})$ is trivial.
\item Let $H^2(X,\mathbb{Z})\rightarrow \Lambda_{K3}$ (resp. $\Lambda_{Ab}$) be a primitive embedding and let us extend the action of $G$ on $\Lambda_{K3}$ (resp. $\Lambda_{Ab}$) trivially on the orthogonal complement of the image of the previous embedding. Then the $G$ invariant part of $\Lambda_{K3}$ (resp. $\Lambda_{Ab}$) contains primitively a lattice isometric to $U(k)$.  
\end{itemize}
Then we call the group $G$ a numerically induced group.
\end{defn} 
\begin{thm}\label{thm:twistinduced}
Let $X$ be a projective manifold of \kntiposp and let $G\subset \Aut(X)$ be a numerically induced group of automorphisms. Then there exist a $K3$ surface $S$, a Brauer class $\beta$ such that $G\subset \Aut(S,\beta)$, and a $G$-invariant twisted Mukai vector $v_B$ such that $X=M_{v_B}(S,\beta)$ and $G\subset \Aut(X)$ is induced from $G\subset \Aut(S,\beta)$.
\end{thm}
\begin{proof}
As $G$ is numerically induced, we have a well defined lattice $P:=U(k)^\perp\subset \Lambda_{K3}(X)$. By the surjectivity of the twisted period map \cite[Prop. 2.8]{H}, there exists a $K3$ surface $S$ and a Brauer class $\beta$ such that $\Lambda_{K3}^{1,1}(X)=\Pic(S,\beta)$. We want to prove that $(S,\beta)$ has a natural $G$-action: by the global twisted Torelli theorem for $K3$'s \cite[Cor. 5.4]{H2}, $G$ acts as an isometry group of the untwisted $H^2(S,\mathbb{Z})$. Therefore, it suffices to find a $G$-invariant \kahl class on $S$. We split the proof in two cases, the first is when $G$ is symplectic and the second is the opposite case (purely non symplectic). The general case can be treated by looking first at the symplectic subgroup of $G$.\\
Let $G$ be symplectic. As $G$ is numerically induced, the covariant part $S_G(X)$ of the $G$ action lies inside $H^2(X,\ZZ)$. As $G$ is symplectic, $S_G(X)$ is contained in $\Pic(X)$, see \cite{Beauville2}. As proven in \cite{m_invol}, $S_G(X)$ does not contain any class of square $-2$, as by \cite[Section 9]{mark_tor}, divisors of square $-2$ always have an effective multiple, thus $G$ invariant \kahl classes of $X$ cannot be orthogonal to them. Finally, as $P^\perp$ is $G$ invariant, $S_G(X)$ lies also in $\Pic(S)$, hence $S$ has a $G$ invariant \kahl class in $S_G(X)^\perp$.  
Let now $G$ be non symplectic. This means, again by \cite{Beauville2}, that the invariant lattice $T_G(X)$ is contained in $\Lambda_{K3}^{1,1}(X)$. Therefore, the $G$ invariant part of the action on $H^2(S,\ZZ)$ is contained in $\Pic(S)$ and we have $S_G(S)=S_G(X)$. We want to prove that there are ample invariant classes in $\Pic(S)$, which is equivalent to prove that there are no $-2$ classes in $\Pic(S)\cap S_G(S)$. Suppose on the contrary that $D\in \Pic(S)\cap S_G(S)$ and $D^2=-2.$ As $S_G(S)=S_G(X)$, we have that $D\in S_G(X)\cap \Pic(X)$ (algebraic classes are kept algebraic). However, this implies that $D$ has an effective multiple by \cite[Section 9]{mark_tor}, therefore $T_G(X)$ cannot contain ample classes, which is absurd.
\end{proof}
The above theorem and its proof hold, word by word, also for generalized Kummer's. Actually, in that case it turns out to be even simpler, as any positive divisor on an abelian surface is either ample or anti ample, without having to check orthogonality to any negative class. If we apply it to the classification of automorphisms of generalized Kummer manifolds contained in \cite{MTW}, we have the following:
\begin{cor}
Let $X$ be a fourfold of Kummer type and let $\mathbb{Z}_{/p\Z}\subset \Aut(X)$ be a non symplectic group of automorphism with $p$ prime. Then it is induced, unless $p=7$, or $p=2,3$ and the invariant lattice on $X$ is one dimensional, or $p=2$ and the action is non trivial on the discriminant group.
\end{cor}
\begin{proof}
When the action is non trivial on the discriminant group, the automorphisms cannot be numerically induced, so these cases (which occur only for $p=2$) have to be excluded. The rest of the proof requires only to look at the lists of invariant lattices in \cite[Section 2]{MTW} and see which of them contains $U(k)$ for some $k$.
\end{proof}

\section{Moduli spaces of IHS's with involutions}\label{moduli}

In order to formulate our results we need to recall some results on moduli spaces.
We want to study moduli spaces of fourfolds of $K3^{[2]}$-type with a non symplectic involution of prescribed type, either with invariant sublattice $T=U(2)\oplus E_8(-2)$, $T=U\oplus E_8(-2)$ or $T=U(2)\oplus D_4(-1)$.

Let $L:=U^{\oplus 3}\oplus E_8(-1)^{\oplus 2}\oplus\langle -2\rangle$ be the Beauville--Bogomolov--Fujiki lattice.
Let us recall what is known about these moduli spaces, by the work of Joumaah \cite[\S 6]{Joumaah}. Let $\mathcal{M}_L$ be a connected component of the moduli space of marked pairs of fourfolds of $K3^{[2]}$-type and let $\mathcal{P}$ be the period map from $\mathcal{M}_L$ to the period domain
\[
D_L:=\left\lbrace x\in \PP(L\otimes \C)|\ q(x)=0, (x,\overline{x})>0\right\rbrace.
\]
Let $(r,a,\delta)$ be the triple of invariants associated to the $2$-elementary lattice $T$: it is $(10,10,0)$ for $T=U(2)\oplus E_8(-2)$, $(10,8,0)$ for $T=U\oplus E_8(-2)$ and $(6,4,0)$ for $T=U(2)\oplus D_4(-1)$. In the sequel $T$ will denote one of these three hyperbolic $2$-elementary sublattices. We fix a primitive embedding $j:T\subset L$, identifying $T$ with its image inside $L$; let $Z$ be its orthogonal in $L$ (usually denoted by $S$ in the literature on automorphisms) and consider the involution $\rho\in \Mon^2(L)$ such that $L^{\rho}=T$, i.e. the extension to $L$ of $\id_T\oplus(-\id_Z)$.

Given $X$ a fourfold of $K3^{[2]}$-type and $\iota\in\Aut(X)$ a non symplectic involution acting on it, Joumaah says that the pair $(X,\iota)$ is \emph{of type $T$} if there exists a marking (also said a $(\rho,T)$\emph{-polarization}) $\eta: H^2(X,\mathbb{Z})\rightarrow L$ such that $\iota^\ast=\eta^{-1}\circ  \rho\circ\eta$. Two pairs $(X_1,\iota_1)$ and $(X_2,\iota_2)$ of type $T$ are isomorphic if there exists an isomorphism $f:X_1\rightarrow X_2$ such that $\iota_2\circ f= f\circ \iota_1$; monodromy operators corresponding to these isomorphisms of pairs are exactly those isometries $g\in\Mon^2(L)$ such that $g\circ \rho=\rho\circ g$ - we denote by $\Mon^2(L,T)$ this subgroup. All these monodromy operators leave $T$ invariant and hence there are well-defined restriction maps $\Mon^2(L,T)\rightarrow O(Z)$ and $\Mon^2(L,T)\rightarrow O(T)$; we denote respectively by $\Gamma_Z$ and $\Gamma_T$ their images. Local deformations of a pair $(X,\iota)$ of type $T$ are parametrized by $H^{1,1}(X)^{\iota}$ (see \
cite[Theorem 2]{B} and \cite[\S 4]{BCS} for more details).

\begin{rem}\label{rem:diff-isom}
The notion of $(\rho, T)$-polarization that we adopt here following \cite{Joumaah} is slightly weaker than the one of \cite{Dol-Kondo} and of \cite{BCS2}; in particular, the two pairs are not necessarily isomorphic as ample lattice polarized pairs.
\end{rem}

One can consider the subspace $\mathcal{M}_{T}$ of $(\rho,T)$-polarized marked pairs $(X,\eta)\in\mathcal{M}_L$. Clearly, for any $(X,\eta)\in\mathcal{M}_T$, we have $$[\eta(H^{2,0}(X))]\in D_Z:=D_L\cap \mathbb{P}(Z\otimes \mathbb{C}).$$ 

In fact, the period map restricts to a holomorphic surjective map 
$$\mathcal{P}: \mathcal{M}_{T}\longrightarrow D_Z^0:= D_Z\setminus \bigcup_{\delta\in\Delta(Z)}(\delta^{\perp}\cap D_Z),$$ where $\Delta(Z)$ is the set of elements of $Z$ either of square $-2$ or of square $-10$ with divisibility $2$ in $L$.

The construction could a priori depend on the choices made, first of all by the $\Mon^2(L,T)$-orbit of the fixed embedding $T\subset L$. The following statement shows that this is not the case here.

\begin{lem}\label{lem:connected}
Let $T$ be one of the three lattices $U(2)\oplus E_8(-2)$, $U\oplus E_8(-2)$ or $U(2)\oplus D_4(-1)$. Then all embeddings of $T$ in $L$ are monodromy equivalent.
\end{lem}
\begin{proof}
We need to prove that two different embeddings are conjugate by a monodromy operator in $L$ which, by the work of Markman \cite{mark_tor}, means that two different embeddings must be conjugate by an orientation preserving isometry of $L$. By composing with the isometry $-\id$ on $L$, this is also equivalent to being conjugate by any isometry. But the proof of this fact is precisely the content of \cite[Proposition 8.2/Remark 8.3]{BCS}. 
\end{proof}

In general, $\mathcal{P}$ is not a local isomorphism and two elements of $\mathcal{M}_{T}$, even in the same fiber of $\mathcal{P}$, are not necessarily deformation equivalent as pairs of type $T$; equivalence classes for deformation of pairs of type $T$ are determined by the choice of a $\Gamma_T$-orbit $\mathcal{O}$ of chambers $K$ inside $C_T\setminus\cup_{\delta\in\Delta(T)}\delta^{\perp}$, where $C_T$ is one connected component of the positive cone of $T\otimes\IR$ and $\Delta(T)$ is the set of elements of $T$ either of square $-2$ or of square $-10$ with divisibility $2$ in $L$.

\begin{lemm}\label{chamber-orbits}
 We have 
 $$\Delta(T)=
 \left\lbrace
 \begin{array}{ccc}
 \emptyset&\mathrm{if}&T=U(2)\oplus E_8(-2),\\
 \lbrace \delta\in T|q(\delta)=-2\rbrace&\mathrm{if}&T=U(2)\oplus D_4(-1)\ \mathrm{or}\ U\oplus E_8(-2).
 \end{array}
\right.
 $$
Moreover, in all cases, $\Gamma_T=\widetilde{O}(T):=\lbrace g\in O(T)| g_{|A_T}=\id_{A_T}\rbrace$ and all pairs of type $T$ are deformation equivalent for each $T$.
\end{lemm}
\begin{proof}
If $T=U(2)\oplus E_8(-2)$, $\Delta(T)=\emptyset$ because any element in $T$ has square a multiple of four. 

If $T=U(2)\oplus D_4(-1)$, it admits only one primitive embedding in $L$ up to isometry (see \cite[Remark 8.3]{BCS}). Consider for example the embedding $j_1:U(2)\subset U\oplus U$ via $t_1\mapsto l_{1,1}+2l_{2,1}$, $t_2\mapsto l_{2,1}+2l_{2,2}$, where $t_1,t_2$ are generators of $U(2)$ in $T$ and $l_{i,j}$ are the standard basis of the $i$-th copy of $U$ for $i,j=1,2$; it is easy to see that $\mathrm{div}_Lt_1=\mathrm{div}_Lt_2=1$. Analogously, one can embed $D_4(-1)$ inside one of the $ E_8(-1)$ summands via a primitive embedding $j_2$, check that all its generators have divisibility one in $L$ and conclude by observing that a primitive embedding of $T$ in $L$ is given by $j_1\oplus j_2$ and that with such an embedding all primitive elements of $T$ have divisibility one inside $L$.

Analogously, for $T=U\oplus E_8(-2)$ it is enough to consider one primitive embedding given by $T\subset L_{K3}:=U^{\oplus 3}\oplus E_8(-1)^{\oplus 2}\subset L$ to conclude that all primitive elements of $T$ have divisibility $1$ inside $L$.

As for the second statement, consider the isotropic subgroup $H:=L/(T\oplus Z)$ of $A_T\oplus A_Z$ which is determined by the overlattice $L$; it is well-known that $\phi\in O(T)$ and $\psi\in O(Z)$ glue to an isometry of $L$ if and only if they act as the identity on the two projections $p_T(H)$ and $p_Z(H)$ of $H$ respectively in $A_T$ and $A_Z$ (see for example \cite[Lemma 3.2]{GHS}). In all three cases, $|H|=|A_T|$, hence $p_T(H)=A_T$, since $p_T$ is a monomorphism. This shows that $\Gamma_T\subset\widetilde{O}(T)$; for any $g\in\widetilde{O}(T)$ on the other hand it is easy to construct an element $G$ of $\Mon^2(L,T)$ restricting to it. Indeed, depending on whether the spinor norm of $g$ in $T$ is $\pm1$, we take $G$ as the extension to $L$ either of the isometry $g\oplus\id_Z$ or of $g\oplus\rho_{w}$, where $\rho_{w}$ is the reflection in a vector $w\in U\subset Z$ such that $w^2=2$.

The last part of the statement now follows from \cite[Theorem 9.11]{Joumaah}: it is a trivial consequence for $T=U(2)\oplus E_8(-2)$; when $T=U(2)\oplus D_4(-1)$ or $T=U\oplus E_8(-2)$, this comes from the fact that there is only one $\Gamma_T$-orbit for chambers of $C_T$ cut out by $\Delta(T)$, as the movable cone coincides with the ample cone by \cite{bht} and \cite{m_wall} and the movable cone is a fundamental chamber for the action of reflections by uniruled negative divisors ($-2$ classes in this case) by \cite[Section 6]{mark_tor}. 
\end{proof}

Following Joumaah in loc.cit., we now introduce the following notion. For the sake of simplicity, we give all definitions only in the case in which there is only one deformation type of pairs of type $T$.

\begin{defn}
A pair $(X,\iota)$ of type $T$ is {\it simple} if the invariant part $\mathcal{K}_X^{\iota}$ of the K\"ahler cone of $X$ is a connected component of the chamber decomposition 
\begin{equation}\label{invar-dec}
\mathcal{C}_X^{\iota}\setminus\bigcup_{D\in\Delta^{\iota}(X)}D^{\perp},
\end{equation}
 where $\mathcal{C}_X^{\iota}$ is the invariant part of the positive cone of $X$ and $\Delta^{\iota}(X)$ is the set of invariant wall divisors $D\in\Pic(X)$.
\end{defn}

In fact, it can happen that $\mathcal{K}_X^{\iota}$ is strictly contained in a chamber of (\ref{invar-dec}). Fix one chamber $K$ inside $C_T\setminus\cup_{\delta\in\Delta(T)}\delta^{\perp}$; the choice of a chamber $K$ identifies the choice of one connected component $D^+_Z$ of $D^0_Z$. Let $\Gamma_{Z,K}\subset \Gamma_Z$ be the restriction to $Z$ of the subgroup of $\Mon^2(L,T)$ of isometries leaving $K$ invariant.  

There is a subset $\Delta(K)\subset L$ such that the hyperplane arrangement $\mathcal{H}_K:=\bigcup_{\delta\in\Delta(K)}\delta^{\perp}$ is locally finite for the action of $\Gamma_{Z,K}$ and such that 
$$\Omega^s_{r,a,\delta}:=(D^+_Z\setminus \mathcal{H}_K)/\Gamma_{Z,K}$$
does not depend on the choice of $K$ inside its $\Gamma_T$ orbit and it is a coarse moduli space for simple pairs of type $T$ (\cite[Theorem 10.5]{Joumaah}). Given the set $\mathcal{M}_{r,a,\delta}^s$ of simple pairs of type $T$, the period map $\mathcal{P}_{r,a,\delta}:\mathcal{M}_{r,a,\delta}^s\rightarrow\Omega^s_{r,a,\delta}$ given by $(X,\iota)\mapsto [\eta(H^{2,0}(X)]$, with $\eta$ a $(\rho,T)$-polarization such that $\eta(\mathcal{K}_X^{\iota})=K$, is bijective and induces the structure of variety on $\mathcal{M}_{r,a,\delta}^s$.

\begin{cor}\label{connectedness}
The moduli spaces $\mathcal{M}_{10,10,0}^s$, $\mathcal{M}_{10,8,0}^s$  and $\mathcal{M}_{6,4,0}^s$ are normal irreducible quasi-projective varieties, respectively of dimension $11$, $11$ and $15$.
\end{cor}
\begin{proof}
It follows from Lemma \ref{chamber-orbits} and Joumaah's work.
\end{proof}
\begin{rem}
In our two cases, it is clear that any fourfold $X$ of $K3^{[2]}$-type with $\Pic(X)\cong T$ carries a simple non symplectic involution of type $T$, see for example \cite[Proposition 8.5]{BCS} for a proof of this fact: this happens exactly in the complex open set $$\mathcal{P}^{-1}\left(D_Z\setminus \bigcup_{\delta\in Z}(\delta^{\perp}\cap D_Z)\right)\subset \mathcal{M}_T.$$ This, together with Proposition \ref{connectedness}, is the reason why we can formulate Theorem \ref{main} in that form.
\end{rem}

\begin{rem}
To give a complete list of families of fourfolds of $K3^{[n]}$ type with a non symplectic involution, we need to analyze the special cases obtained by Ferretti \cite{Ferretti}. He produced special EPW sextics with contractible divisors and an involution which has invariant lattice of the form $\langle 2\rangle\oplus \langle-2\rangle^a$, for all $a\leq 10$, where all primitive elements have divisibility $1$. Therefore, the movable cone coincides with the ample cone and the moduli space is connected as in the case of $U(2)\oplus D_4(-1)$.
\end{rem}

\section{Double EPW quartic sections as moduli spaces of twisted sheaves}\label{derived}

Let $V= Q\cap C(\PP^2\times \PP^2)=Q\cap C(\PP(U_1 )\times \PP(\wedge^2 U_2 )) $ be a fixed Verra fourfold and let $k$ 
be a linear system of hyperplanes on $V\subset \PP^9$. The projections $\pi_i \colon V\to \PP^2$ for 
$i=1,2$ define two quadric bundle structures, given by two divisors $h$ and $h_2$ (such that $h+h_2=k$) on $V$ with 
discriminant the sextic curves $C_1, C_2\subset \PP^2$ respectively. 
Denote by $S_1, S_2$ the K3 surfaces being $2:1$ covers of $\PP^2$ branched along $C_1$ and $ C_2$ respectively and by $H$ the degree $2$ polarisation inducing $S_1\to \PP^2$. 
Since smooth quadrics admit two rulings by $\PP^1$, the quadric bundle $Z$ induces natural $\PP^1$ fibrations on $S_1$ and $S_2$. 
Thus two-torsion Brauer classes $\beta_1$ in $\Br_2(S_1)$ and $\beta_2$ in $\Br_2(S_2)$ (\cite{Yo}).

The map $\varphi\colon F(V)\to X_V\to D_V$ factorizes by a $\PP^1$ fibration $\phi$ and a $2:1$ cover 
such that $X_V$ is an IHS fourfold with Beauville-Bogomolov degree $4$ as discussed in Section
\ref{EPW quartic}. 
It was also proven in \cite{IKKR} that a general $X_V$ is isomorphic to a moduli space of twisted sheaves on some K3 surface.
The aim of this section is to make this relation more explicit.

\begin{thm}\label{moduli-twisted} If $V$ is smooth then 
the IHS fourfold $X_V$ is birational to the moduli space of sheaves 
on $S_1$ (resp. $S_2$) with Brauer class $\beta_1$ (resp.~$\beta_2$) 
and Mukai vector $(0,H,0)$.
\end{thm}
We expect that $X_V$ is actually isomorphic to the moduli space of twisted sheaves above.
We will not need this stronger version since the considered IHS fourfolds (with the given lattices) do not admit other birational models.
We shall in fact describe how to associate to a $(1,1)$ 
conic on $V$ a twisted sheaf on $S_1$ (or $S_2$) (cf.~\cite{Kuz,MS}).
Before the proof of our theorem we prove some technical results.

Throughout the proof we will fix one of the two projections, $\pi_1$, having as discriminant curve the 
smooth sextic $C_1$. Consider the projective bundle
$$q:\bar{\PP}:=\PP(3(\oo_{\PP^2}(1))\oplus \oo_{\PP^2})\to \PP^2$$ being a natural desingularisation 
$\beta$ of $C(\PP^2\times \PP^2)$. Denote  by $\alpha$ the natural embedding 
of $V$ in 
$\PP(3(\oo_{\PP^2}(1))\oplus \oo_{\PP^2})$, such that $\pi_1=q\circ 
\alpha $.

Since $V$ is a quadric fibration one can use the results from \cite[Thm.~4.2]{Kuz} showing that 
the derived category admits a semiorthogonal decomposition 
\begin{equation}\label{qA}D^b(V)=<\Phi(D^b(\PP^2,\mathcal{B}_0)),\oo_V(-h),\oo_V,\oo_V(h),\oo_V(-h+k),\oo_V(k),\oo_V(h+k)>\end{equation}
where:
\begin{itemize}
\item  $\mathcal{B}_0$ is the sheaf of even parts of the Clifford algebra (and $\mathcal{B}_1$ of the odd 
parts),
\item $ D^b(\PP^2,\mathcal{B}_0)$ is equivalent to the category of the twisted K3 surface 
$D^b(S_1,\beta_1)$
through an equivalence described in \cite[Lem.~4.2]{Ku2},
\item $\Phi\colon D^b(\PP^2, \mathcal{B}_0)\to D^b(V)$ is a fully faithful functor (see \cite[Prop.~4.9]{Kuz}) such that $$\Phi(\mathcal{A})\colon= \pi_1^{\ast}
\mathcal{A}\otimes_{\pi_1^{\ast}\mathcal{B}_0} \mathcal{E}'$$
\item $\mathcal{E}'$ is a rank $4$ vector bundle on $V$,
with a natural structure of a right $\pi_1^{\ast} \mathcal{B}_0$ module, with a resolution 
$$0 \to q^{\ast} \mathcal{B}_0(-2k) \to q^{\ast} \mathcal{B}_1(-k) \to \alpha_{\ast} \mathcal{E}' \to 0 .$$
\end{itemize}

Denote by $pr\colon D^b(V)\to D^b(\PP^2,\mathcal{B}_0)$ the adjoint functor of $\Phi\circ [1]$: it is given by
\begin{equation}
pr(-)=\pi_{1\ast}((-)\otimes \oo_V(3h)\otimes \mathcal{E}[1])
\end{equation}
where $\mathcal{E}$ satisfies (see \cite[Eq.(15), Lem.~4.10]{Kuz})
$$0 \to q^{\ast} \mathcal{B}_1(-3k-h) \to q^{\ast} \mathcal{B}_0(-2k) \to \alpha_{\ast} \mathcal{E} \to 0 .$$
It follows from the decomposition (\ref{qA}) that as adjoint to $\Phi\circ [1]$ the functor  $pr$ can be presented as a composition of left mutations:
\begin{equation}\label{mutation pr} pr=L_{-h}\circ L_0\circ L_{h}\circ L_{-h+k}\circ L_{k}\circ L_{h+k},
\end{equation}

where 
$L_{mh+ak}$ denote the left mutation with respect to $\mathcal{O}_V(mh+ak)$. In particular, 
\begin{equation}\label{vanishing pr}
pr(<\oo_V(-h),\oo_V,\oo_V(h),\oo_V(-h+k),\oo_V(k),\oo_V(h+k)>)=0.
\end{equation}

Let $c\subset V$ be a $(1,1)$ conic that defines two lines $l_1,l_2$ on the $\PP^2$ bases. 
Let $Q_{l_1,l_2}$ be the quadric threefold being the cone 
$$C(l_1\times l_2) \subset C(\PP^2\times \PP^2)\subset \PP^9$$ and $D_c:=Q_{l_1,l_2}\cap Q$ 
the corresponding del Pezzo surface of degree $4$. 
It was proven in \cite[\S 3]{IKKR} that for a generic choice of $c$ the surface $D_c$ is smooth and that the fibers of $\phi$ are linear 
systems of appropriate conics on $D_c$ so the fiber is determined by $\mathcal{I}_{c|D_c}(k)$
(see \cite[\S 3]{IKKR} for the case when $D_c$ is singular).
We are interested in the image 
$pr( \mathcal{I}_{c| V} (k) )$ and shall show that it is a twisted sheaf on $S_1$ with the Brauer class $\beta_1$.

Let $f:S_1\to \PP^2$ be the double cover.
Recall that from \cite[Lem.~4.2]{Ku2} we have an equivalence
$D^b(S_1,\beta_1) \simeq D^b(\PP^2,\mathcal{B}_0)  $ given by $\Gamma(F)=f_*F$.
Consider the functor $$\Xi\colon F(V)\to D^b(S_1,\beta_1)$$ being the composition of $\Gamma^{-1}$ with the following functor
$$ F(V)\ni c\to pr(\mathcal{I}_{c|V}(k))\in D^b(\PP^2,\mathcal{B}_0).$$
We shall show that $\Xi$ induces a birational map between $X_V$ 
and the moduli space of twisted sheaves of $S_1$ with 
Brauer class $\beta_1$  and Mukai vector $(0,H,0)$. 
In order to prove Theorem \ref{moduli-twisted} we need the following lemma. 
\begin{lem}\label{lem:twistedmoduli} 
For a general $(1,1)$ conic $c\subset Y$, the element  $\Xi(\mathcal{I}_{c| V} (k) )
=\Xi(\mathcal{I}_{c|D_c}(k))$ is a twisted sheaf on $S_1$ inside 
$D^b(S_1,\beta_1)$. 
\end{lem}
\begin{proof} It is enough to show the statement for $pr(\mathcal{I}_{c|V}(k))$.
Let $H_1$ and $H_2$ be $(1,0)$ and $(0,1)$ divisors on $V$ such that 
$D_c=H_1\cap H_2$. 

By the exact sequence
$$0\to \mathcal{I}_{D_c| V} (k) \to \mathcal{I}_{c| V}(k)\to \mathcal{I}_{c|D_c} 
(k) \to0$$
it is enough to prove that $pr( \mathcal{I}_{D_c| V} (k))=0$.
The last follows from (\ref{vanishing pr}) and the short exact sequence
$$0\to \oo_V \to \oo_V(1,0)\oplus \oo_V(0,1) \to \mathcal{I}_{D_c| V} (k) \to 0,$$ thus we get $pr(\mathcal{I}_{c| V} (k) )=pr(\mathcal{I}_{c|D_c}(k))$.

From (\ref{vanishing pr}) we also have $pr(\oo_V)=pr(\oo_V(H_2))=pr(\oo_V(H_1+H_2))=0$. 
So using the sequences

$$0\to \oo_{H_2}(H_2)\to \oo_{H_2}(H_1+H_2) \to \oo_{D_c}(k)\to 0,$$
$$0\to \oo_V(-H_2)\to \oo_V \to \oo_{H_2}\to 0,$$
we infer $pr(\oo_{D_c}(k))=0$. 
Now from

$$0\to \mathcal{I}_{c|D_c}(k)\to \oo_{D_c}(k)\to \oo_c(k)\to 0,$$
we obtain 
$$pr(\mathcal{I}_{c|D_c}(k))=
pr(\oo_c(k))[-1]={\pi_1}_{\ast}(\oo_c(k)[-1]\otimes 
\oo_V(3h)\otimes \mathcal{E}[1])=$$
$$={\pi_1}_{\ast}(\mathcal{E}(k)|_c)\otimes\oo_{\PP^2}(3).$$
\end{proof}

\begin{lem}\label{lem:injectivity}
Two general $(1,1)$ conics $c_1,c_2\subset V$ are in the same fiber of $\phi$ iff 
$$pr( \mathcal{I}_{c_1| V} (k) )=pr( \mathcal{I}_{c_2| V} (k)).
$$ 
\end{lem}
\begin{proof}
In order to prove the statement, notice first that
$$\mathcal{I}_{c|D_c}(k)\in <\oo_V(h),\oo_V(-h+k),\oo_V(k),\oo_V(h+k),\oo_V(-h+2k)>^{\perp}.$$
Indeed  it is enough to prove that $$H^*(\mathcal{I}_{c|D_c}(k-h))=H^*(\mathcal{I}_{c|D_c}(h))=H^*(\mathcal{I}_{c|D_c})=H^*(\mathcal{I}_{c|D_c}(-h))=H^*(\mathcal{I}_{c|D_c}(h-k))=0.$$ Since $c$ is general, $D_c$ is a del Pezzo surface of degree 4 i.e.~a blow up of $\mathbb{P}^2$ in 5 points in general position, its Picard group is generated by the pullback $H$ of the hyperplane class on $\mathbb{P}^2$ and the exceptional divisors $E_1,\dots E_5$. We have $\oo_V(k)|_{D_c}=3H-E_1-\dots E_5$ and without loss of generality we can assume $\oo_V(h)|_{D_c}=H-E_1$ and $c\in |H-E_2|$. We can now easily compute 
$$H^0(\mathcal{I}_{c|D_c}(k-h))=H^0(\mathcal{I}_{c|D_c}(h))=H^0(\mathcal{I}_{c|D_c})=H^0(\mathcal{I}_{c|D_c}(-h))=H^0(\mathcal{I}_{c|D_c}(h-k))=0$$ and by Serre duality 
$$H^2(\mathcal{I}_{c|D_c}(k-h))=H^2(\mathcal{I}_{c|D_c}(h))=H^2(\mathcal{I}_{c|D_c})=H^2(\mathcal{I}_{c|D_c}(-h))=H^2(\mathcal{I}_{c|D_c}(h-k))=
0$$ and we conclude by Riemann-Roch theorem that $$H^1(\mathcal{I}_{c|D_c}(k-h))=H^1(\mathcal{I}_{c|D_c}(h))=H^1(\mathcal{I}_{c|D_c})=H^1(\mathcal{I}_{c|D_c}(-h))=H^1(\mathcal{I}_{c|D_c}(h-k))=0.$$
Consider now the following
 \begin{equation} \label{sequence Fc}
0\to F_c \to 2\oo_V \to \mathcal{I}_{c| D_c} (k) \to 0,
\end{equation}

where $F_c:=L_0(\mathcal{I}_{c|D_c}(k))$ is a rank two sheaf. 
By the above we have $pr(\mathcal{I}_{c|D_c}(k))=L_{-h}(F_c)$ and 
$$F_c\in <\oo_V,\oo_V(h),\oo_V(-h+k),\oo_V(k),\oo_V(h+k)>^{\perp}.$$

Note also that $\mathcal{I}_{c|D_c}(k)$ can be reconstructed from the sheaf $F_c$ 
by taking the bi-dual of  the above sequence (cf.~\cite[p.4]{AL}). 

Finally we observe that $L_{-h}$ defines a derived equivalence 
$$<\oo_V,\oo_V(h),\oo_V(-h+k),\oo_V(k),\oo_V(h+k), \oo_V(2k-h)>^{\perp} \rightarrow $$ 
$$<\oo_V(-h),\oo_V,\oo_V(h),\oo_V(-h+k),\oo_V(k),\oo_V(h+k)>^{\perp},$$ whose inverse is the right mutation
$R_{-h}$. It follows that  $\mathcal{I}_{c|D_c}(k)$ is uniquely determined by $pr(\mathcal{I}_{c|D_c}(k))=L_{-h} (F_C)$.
\end{proof}

\subsection{The proof of Theorem \ref{moduli-twisted}}
By Lemma \ref{lem:twistedmoduli} and by tensoring with an appropriate power 
of $H$, one
deduces that there exists a birational map 
$$F(V)=Hilb_{(1,1)}(V) \dasharrow  M_v(S_1,\beta_1),$$ where $v = (0, G, a)$ and
$a \in \{0, 1\}$ and $G\in \Pic(S_1)$ (the rank $0$ part is obvious). Indeed, 
these objects are also stable as their support is on a curve, which is generically smooth, and for a general $K3$ surface $S_1$ there are no strictly semistable objects with a primitive Mukai vector. 
Notice that $\oo_c(k)$ has primitive class in the Grothendieck $K$ ring, hence 
so does its projection in the twisted $K3$ category (in general one can argue as in
\cite[Prop.~3.4]{MS}).
Thus, the Mukai vector of the projection is of the form $(0,H,a)$ for some $a$.

Now, Lemma \ref{lem:twistedmoduli} proves that there is a rational map from $X_V$ 
into $M_{(0,H,0)}(S_1,\beta_1)$ and Lemma \ref{lem:injectivity} proves 
that the map is finite and dominant, hence we have that the two spaces are 
isomorphic by smoothness of $X_V$.

As proven in \cite{IKKR}, $\Pic(X_V)$ for general $V$ 
contains only elements of square which is a multiple of $4$. 
If $a=1$, $\Pic(M_{(0,H,1)}(S_1,\beta_1))$ contains the class of $H-[S_1]$, 
which has square $2$. Hence, $a=0$.\qed

\section{IHS fourfolds with Picard lattice $U(2)\oplus E_8(-2)$}

In this section we will construct eleven dimensional families of symmetric fourfolds of $K3^{[2]}$-type, 
first as the family of double EPW quartics possessing an additional symplectic 
involution and then  as a family of moduli spaces of twisted sheaves on a $K3$ surface. 

\subsection{Verra fourfolds with involution}  \label{Verras}
Let us study Verra fourfolds that are symmetric with respect to a linear involution of $C(\mathbb{P}^2\times \mathbb{P}^2)$.
We shall see that there are three families of such fourfolds. Let us denote by $V'\subset \mathbb{P}^2\times \mathbb{P}^2$ a Verra threefold associated to a symmetric Verra fourfold $V$. Clearly $V'$ is also symmetric.

Furthermore, let us observe that the action on 
$\mathbb{P}^2\times \mathbb{P}^2$ inducing the action on $V'$ is of the form $\iota=(\iota_1, \iota_2)$ such that one of the following holds:
\begin{enumerate}
\item $\iota_1=\id$
\item $\iota_2=\id$
\item $\iota_i$ are both nontrivial 
\end{enumerate}
The symmetric Verra 3-folds are characterized by the fact that they are invariant under the action of $\iota$ on $\mathbb{P}^2\times \mathbb{P}^2$. Furthermore, the quadric $q$ defining $V'$ can be chosen to be invariant under the action induced by $\iota$ on $\mathbb{P}^8\supset \mathbb{P}^2\times\mathbb{P}^2$. This is because only in those cases the involution $\iota$ lifts to the double cover of $\mathbb{P}^2\times \mathbb{P}^2$ branched over the Verra 3-fold $Q\cap ( \mathbb{P}^2\times \mathbb{P}^2)$ i.e. the Verra 4-fold. 
Let us denote by $(x_1,x_2,x_3),(y_1,y_2,y_3)$ the coordinates on $\mathbb{P}^2\times \mathbb{P}^2$;  without loss of generality we may assume that we are in one of the following possibilities:
\begin{enumerate}
\item $\iota((x_1,x_2,x_3),(y_1,y_2,y_3))=((x_1, x_2,x_3),(-y_1,y_2,y_3))$
\item $\iota((x_1,x_2,x_3),(y_1,y_2,y_3))=((-x_1,x_2,x_3),( y_1,y_2,y_3))$
\item$\iota((x_1,x_2,x_3),(y_1,y_2,y_3))=((-x_1, x_2,x_3),(-y_1,y_2,y_3))$
\end{enumerate}

Let us write down  in each case the divisors of type $(2,2)$ on $\mathbb{P}^2\times \mathbb{P}^2$ invariant under $\iota$. Note that in each of these cases the involution $\iota$ admits two lift to automorphism of the Verra fourfold $V$. These involutions commute with the covering involution which is their composition.
In each case, the Verra 3-fold is described by an equation of degree $(2,2)$ given by 
$$\sum_{
(i_1\dots i_3,  j_1\dots j_3) \in A 
}  a_{(i_1\dots i_3,  j_1\dots j_3)} x_1^{i_1}x_2^{i_2}x_3^{i_3}y_1^{j_1}y_2^{j_2}y_3^{j_3}
$$
with $A\subset \mathbb{N}^6$ given by:
\begin{enumerate}
\item $A=\{(i_1\dots i_3,  j_1\dots j_3) \in \mathbb{N}^6: i_1+i_2+i_3=j_1+j_2+j_3=2, i_1 \in 2\mathbb{N}\}$,
\item $A=\{(i_1\dots i_3,  j_1\dots j_3) \in \mathbb{N}^6: i_1+i_2+i_3=j_1+j_2+j_3=2, j_1 \in 2\mathbb{N}\}$,
\item $A=\{(i_1\dots i_3,  j_1\dots j_3) \in \mathbb{N}^6: i_1+i_2+i_3=j_1+j_2+j_3=2, (i_1+j_1) \in 2\mathbb{N}\}$.
\end{enumerate}
We can now compute the discriminant of both projections to $\mathbb{P}^2$ for the general element of each family. We obtain:

\begin{lemm}\label{discr}
\begin{enumerate}
\item  If a Verra fourfold admits an involution of type $(1)$ (resp.~$(2)$) as above then the projection to $(x_1,x_2, x_3)$  (resp. $(y_1,y_2, y_3)$) has discriminant a smooth symmetric sextic and the projection to $(y_1,y_2, y_3)$ ( resp. $(x_1,x_2, x_3)$) has discriminant being the union of a conic and a quartic. 
\item  If the Verra fourfold admits an involution of type $(3)$, both projections have smooth symmetric sextics as discriminants.
\end{enumerate}
\end{lemm}\begin{proof}This is a straightforward computation using the description of the symmetric Verra threefolds. 
\end{proof}

Let us denote the family of symmetric Verra fourfolds in (1) by $\mathcal{V}_1$ and the one in (3) by $\mathcal{V}_2$.

\begin{prop}\label{families-dimension}
The families $\mathcal{V}_1$, $\mathcal{V}_2$ represent 11 dimensional subsets of the moduli space of linear isomorphism classes of $(2,2)$ divisors on $\mathbb{P}^2\times \mathbb{P}^2$ 
\end{prop}
\begin{proof}
The proof is a simple parameter count. We compute the dimension of the space of polynomials of bi-degree $(2,2) $ invariant with respect to the chosen involution and subtract the dimension of the group of automorphisms of $\mathbb{P}^2\times \mathbb{P}^2$ fixing the involution. We get 24-9-5+1=11 in the first two cases and 20-5-5+1 in the third case.
\end{proof}
\begin{prop}For the families $(1)$ and $(2)$ the quadric fibration of the Verra fourfold over $\mathbb{P}^2$ with discriminant a smooth sextic admits a section.
\end{prop}
\begin{proof}Observe that the fiber of the second projection over the singular point of the reducible sextic is a reducible quadric. Each of its components, which are planes, defines a section of the original projection.
\end{proof}

\begin{cor}
The Verra fourfolds in families $(1)$ and $(2)$ are rational.
\end{cor}
\begin{proof}
It is well-known that quadric surface bundles are rational when they admit a section.
\end{proof}

\begin{rem}
The problem of rationality of Verra fourfolds is analagous to that for cubic fourfolds; these two cases are the only exceptions to the criterion given by Schreieder in his recent preprint \cite[Corollary 2]{Schreieder}.
\end{rem}
Let us now come back to the study of IHS manifolds. We keep our notation where $V$ is a general symmetric Verra fourfold, we know that the variety $X_{V}$ being the base of the natural $\mathbb{P}^1$ fibration on the Hilbert scheme of $(1,1)$ conics on $V$ is an IHS manifold, more precisely a double EPW quartic section. Let us write down how the existence of the involution on $V$ influences the Picard lattice of $X_{V}$.

\begin{prop}\label{prop:e8_lagr}
Let $X_{V}$ be the double EPW quartic section associated to a general symmetric Verra fourfold $V$ in one of the three families above. 
Then $X_{V}$ admits a non symplectic involution which 
leaves the Picard lattice invariant; moreover, $\Pic(X_{V})$ is an over-lattice of $U(2)\oplus E_8(-2)$ which contains primitively $U(2)$ and $E_8(-2)$.
\end{prop}
\begin{proof} 
Each involutions  of the Verra fourfold $V$ corresponding to type (1),(2) and (3), as above, induce a natural involution on the Hilbert scheme of $(1,1)$ conics
on $V$.  Indeed, such an involution preserves the $\PP^1$ fibration and it descends to an involution on $X_{V}$. Furthermore, since the latter involution is induced by an involution on $\mathbb{P}^2\times \mathbb{P}^2$, it commutes with the covering involution $\mathbf{i}$ of $X_V$. It follows that we have three involutions on $X_V$ (note that they correspond to three involutions on the Verra fourfold).
Moreover, since $\mathbf{i}$ is non symplectic, we infer that one of the remaining two involutions is a symplectic involution. Let us denote the symplectic involution by $\mathbf{j}$, the third involution is then $\mathbf{i}\circ\mathbf{j}$.  It is 
well 
known (see \cite[Thm. 5.2]{m_invol}), that $\mathbf{i}$ acts as $-1$ on a lattice 
isometric to $E_8(-2)\subset \Pic(X_V)$ and preserves its orthogonal, whereas $\mathbf{j}$ 
preserves $U(2)\subset \Pic(X_V)$ and acts as $-1$ on its orthogonal. Let us 
consider the non symplectic involution $\mathbf{i}\circ \mathbf{j}$. 

As $\mathbf{j}$ comes from 
an involution on $Y_V$ preserving the projection to $\mathbb{P}^2\times 
\mathbb{P}^2$, it must fix the lattice fixed by $\mathbf{i}$, hence $\mathbf{j}\circ\mathbf{i}$ 
fixes $U(2)\oplus E_8(-2)$ and, by the generality assumption on $V$ and the dimension count of Prop. \ref{families-dimension}, we infer that
$\Pic(X_V)$ is an overlattice of $U(2)\oplus E_8(-2)$. Furthermore, $U(2)$ and $E_8(-2)$ have to be primitive sublattices of $\Pic(X_V)$, since they are the invariant, respectively the co-invariant, sublattice of a non symplectic, respectively symplectic, involution.
\end{proof}

\begin{rem}\label{overlattices}
Every overlattice of $U(2)\oplus E_8(-2)$ that contains primitively $U(2)$ and $E_8(-2)$ is isometric to $$U(2)\oplus E_8(-2),\ \ U\oplus E_8(-2) \ \ \text{or} \ \ U(2)\oplus D_4(-1)^{\oplus 2}.$$
The IHS associated to the first two lattices will be discussed in the next section.
We expect that the IHS fourfolds of $K3^{[2]}$ type corresponding to the third family are $4:1$ covers of 
a quadric section of $C(\PP^2\times \PP^2)\subset \PP^9$.
\end{rem}

\subsection{Symmetric two torsion in the Jacobian of symmetric sextic curves in $\PP^2$}\label{symmetric two torsion}
The aim of this section is to prove that we can obtain a general symmetric sextic as the discriminant locus of a symmetric Verra threefold.
The strategy is to start from a symmetric sextic curve and to construct symmetric Verra fourfolds from it.
The result will be proven in Section \ref{step}.  

First, let us describe our construction of symmetric Verra threefolds from the point of view of \cite{B}. 
 We start with a 
symmetric sextic $s\colon C_1\to \PP^2$ which does not pass through the center of
a given symmetry $j$ in $\PP^2$. Then we consider $\eta$ a symmetric $2$-torsion element 
in the Jacobian $J(C_1)$. Note that from \cite{B}, there are two types of such elements, either they are induced 
from the quotient curve or not (see Section \ref{121}). 
We take $\eta$ symmetric and show that the minimal free resolution of $s_{\ast}(\eta(2))$ gives 
rise to a $3\times 3$ matrix of quadratic forms on $\PP^2$ (\cite{C}). This matrix produces a Verra fourfold, and
thus also an IHS fourfold by Section \ref{EPW quartic}.

 Let us be more precise.
Consider a plane $\mathbb{P}^2$ with involution $j:\mathbb{P}^2\to \mathbb{P}^2$.
Let $\mathcal{C}$ be the family of smooth sextic curves in $ \mathbb{P}^2$ defined by sextic equations invariant (not just equivariant) with respect to $j$. Let $C_1\in \mathcal{C}$ and $S_1$ be the K3 surface obtained as the double cover $\pi:S_1\to\mathbb{P}^2$ branched in $C_1$. Note that under our assumption the involution $j$ lifts to an involution $\tilde{j}:S_1\to S_1$. These symmetric K3 surfaces $S_1$ were studied in \cite[\S 3.2]{vGS}.

\subsubsection{Symmetric $2$-torsion elements in $J(C_1)$}\label{121}
By \cite{vGS}, the quotient of $S_1$ by $\tilde{j}$ is a del Pezzo surface $D$ of degree 1. Let now $\mathcal{W}\subset \operatorname{Pic} C_1$ be the set of 2-torsion elements invariant with respect to $\tilde{j}$. 
Let us first prove the following.
\begin{prop}\label{ala} The following hold:
\begin{enumerate}
\item The set $\mathcal{W}$ has $2^{12}$ elements.
\item The set $\mathcal{V}= \{\eta\in W : H^0(\eta(1))\neq 0\}$ has at most 120 elements. 
\end{enumerate}
\end{prop}
\begin{proof}
For the first part, the symmetry $j$ induces a $2:1$ map $\varphi: C_1\to B$ branched at $6$ points (the intersection of $C_1$ with the fixed line), where $B$ is a curve of genus $4$.
It follows from \cite[Lem.2]{Beauville1} that we have an exact sequence:
\begin{equation}\label{beau} 0 \to \Pic(B)[2]\to(\Pic(C_1)[2])^i \to (\ZZ/2\ZZ)^5\xrightarrow{f} \Pic(B)/2\Pic(B)
\end{equation}
such that the kernel of $f$ is isomorphic to $\ZZ/2\ZZ$.
Thus $(\Pic(C_1)[2])^i$ is a vector space of dimension $12$ over $\ZZ/2Z$.
The image of $\Pic(B)[2]$ forms a subspace of dimension $8$ so we have two sets of non-trivial $2$-torsion elements on
$C_1$ of cardinalities $2^8-1$ and $2^{12}-2^8-1$.

To prove the second part we follow \cite{vGS}. Let us take $\eta \in \mathcal{W}$ such that $H^0(\eta(1))\neq 0$. Let $\xi \in H^0(\eta(1))$, then $\xi^2 \in H^0(\mathcal{O}_{C_1}(2))$. Since we know that the restriction map
$H^0(\mathcal{O}_{\mathbb{P}^2}(2))\to H^0(\mathcal{O}_{C_1}(2))$ is surjective, it follows that there exists a conic $c\subset \mathbb{P}^2$ such that its intersection with $C_1$ is a double divisor. We deduce that the preimage $\pi^{-1}(c)$ decomposes as a union of two curves of self intersection $-2$. Each of these curves must be symmetric with respect to $\tilde{j}$ i.e. they map to $-1$ curves on the del Pezzo surface $D$. The number of $-1$ curves on $D$ is $240$ hence the number of conics $c$ is $120$, which implies the assertion.
\end{proof}

Consider now $\eta\in \mathcal{W}_0=\mathcal{W}\setminus \mathcal{V}$, then by \cite[\S 3]{Vprym} we know that $\eta(2)$
has the following resolution:
$$ 0\to 3\mathcal{O}_{\mathbb{P}^2}(-2)\xrightarrow{\alpha_\eta}  3\mathcal{O}_{\mathbb{P}^2}\to \eta(2)\to 0,$$
with $\alpha_{\eta}$ symmetric. In particular $ \alpha_{\eta}$ defines a $(2,2)$ divisor on $\mathbb{P}^2 \times \mathbb{P}^2$ . We define $V'_{\eta}$ and $V_{\eta}$ as the corresponding Verra threefold and fourfold respectively.
\begin{prop}\label{invV} The fourfold $V_{\eta}$ is a smooth fourfold with projection 
$V_{\eta}\to \mathbb{P}^2$ being a quadric fibration with discriminant curve $C_1$. Furthermore, 
$V_{\eta}$ is invariant with respect to some involution on $\mathbb{P}^2\times \mathbb{P}^2$.
Moreover, the elements from $J(C_1)_2$ induced from the quotient curve give rise to Verra fourfolds from $\V_1$ and the rest give rise to Verra fourfolds in $\V_2$.
\end{prop}

\begin{proof}Clearly $V_{\eta}$ is fibered by quadrics with discriminant locus being the support of $\eta(2)$ i.e.~$C_1$. Since $C_1$ is smooth, it follows that $V'_{\eta}$ is also smooth and as a consequence so is $V_{\eta}$. For its invariance under the involution, consider the resolution:
\begin{equation}\label{eq1}0\to  3\mathcal{O}_{\mathbb{P}^2}(-2)\xrightarrow{\alpha}  3\mathcal{O}_{\mathbb{P}^2}\to \eta(2)\to 0,\end{equation}
and apply $j^*$ to get 
$$0\to  3\mathcal{O}_{\mathbb{P}^2}(-2)\xrightarrow{j^*\alpha}  3\mathcal{O}_{\mathbb{P}^2}\to j^*\eta(2)\to 0,$$
now the identity map $\eta(2)\to j^* \eta(2)=\eta(2)$ lifts to a map between the resolutions. Denote the corresponding maps
$3\mathcal{O}_{\mathbb{P}^2}\to 3\mathcal{O}_{\mathbb{P}^2}$ by $\beta$ and  $3\mathcal{O}
_{\mathbb{P}^2}(-2) \to  3\mathcal{O}_{\mathbb{P}^2}(-2)$ by $\gamma$. Note that $\beta$ is the identity iff 
$\eta$ is coming from the quotient curve. Indeed it follows from \cite[Prop.~1]{B} it act nontrivially on the sections. Since both $\alpha$ and $j^* \alpha$ are symmetric maps we 
infer $\gamma^{-1}= \beta^T$. It is now enough to observe that by applying $j^*$ to the whole diagram we 
obtain that $\beta\circ \beta$ 
is a lift of $\mathrm{id}: \eta(2) \to \eta(2)$. But, since there are no maps $3\mathcal{O}_{\mathbb{P}^2}\to 3\mathcal{O}_{\mathbb{P}^2}(-2)$, the lift of the map  to the resolution is unique. Hence $\beta$ is either an involution  $3\mathcal{O}_{\mathbb{P}^2}\to 3\mathcal{O}_{\mathbb{P}^2}$ or the identity. Note that $\beta$ is in fact equivalent to an involution $\bar{\beta}$ on $H^0( 3\mathcal{O}_{\mathbb{P}^2})$. Now 
observe that we proved that $\alpha=\beta\circ j^*\alpha \circ \beta^T$, which implies that the $(2,2)$ equation defining $V'_{\eta}$ is invariant under the involution $( j,\bar{\beta})$ on $\mathbb{P}^2\times \mathbb{P}^2$. It follows that the involution $( j,\bar{\beta})$ lifts to  $V_{\eta}'$.
\end{proof}

\begin{rem} 
Note that, from \cite[Thm 1.1]{IOOV}, for a smooth
sextic curve $C_1$ we have a natural epimorphism of groups:
\begin{equation}\label{brau}( \frac{\Pic (C_1)}{K_{C_1}})[2] \xrightarrow{\psi} \Br S_1[2] \to 0,\end{equation}
where the group $( \frac{\Pic (C_1)}{K_{C_1}})[2]$ consists of the 2-torsion classes in the 
Jacobian $J(C_1)$ and the theta characteristics and $S_1$ is the double cover of $\PP^2$ branched along $C_1$. 
 One can prove that the symmetric two torsion elements from $J(C_1)_2$, which are induced from the quotient curve, are mapped to the trivial Brauer class by $\psi$.
We expect that the images through $\psi$ (cf.~Equation \ref{brau}) of symmetric two torsion points in the Jacobian $J(C_1)_2$ that are not induced from the quotient curve (cf.~\cite{B} and Equation \ref{beau})  induce non trivial Brauer classes.
\end{rem}

\begin{rem} Note that in \cite{Vprym} Verra considered two torsion elements in the Jacobian of a sextic from $\mathcal{W}-\mathcal{V}$. He shows that such elements are naturally related to quartic double solids with one node. From this point of view it is natural to try to generalize our construction of  double EPW quartics from Section \ref{EPW quartic} in the context of quartic double solids.
\end{rem}

\subsection{Fixed loci of non symplectic involutions and Enriques surfaces}\label{fixed}
Recall that Verra fourfolds are naturally associated to Lagrangian spaces in $\wedge^3 W$ with $W=U_1\oplus U_2$ such that $\mathbb{P}(A)$ meets the Grassmannian $G(3,W)$ only in the point $[U_1]$.
Let us describe  the Lagrangian subspaces defining in this way elements of the considered families of Verra threefolds. 
Let $A\subset \wedge^3 W$ be a Lagrangian space such that $V_{A}$ is a Verra fourfold in one of our families.  Denote by $U_1\subset W$ the space
corresponding to the point  of intersection $\mathbb{P}(A)\cap G(3,W)$ i.e.~$[\wedge^3 U_1]=\mathbb{P}(A)\cap G(3,V)$. Choose $U_2$ in such a way that 
$Q_{A,U_2}$ is a quadric symmetric with respect to the symmetry on $\mathbb{P}(T_{[U_1]})$ induced by the symmetry on the Verra fourfold.

Note that there is an involution $\kappa$ on the vector space $W=U_1\oplus U_2$ inducing  the involution $\iota$ on $U_1\otimes \wedge^2 U_2$.  Let $W^+$ and $W^-$ 
denote the invariant and anti-invariant sub-vector spaces.  Note that in all our cases we can assume $W^-$ is of dimension four whereas $W^+$ is of dimension two.
Finally, let $\tilde{\iota}$ be the involution induced by $\kappa$ on   $\wedge^3 U_1\oplus (\wedge^2U_1\otimes U_2)$.
We observe that $A$ must be symmetric with respect to the involution on $ \wedge^3 W$ induced by $\kappa$ since $Q_A$ is symmetric. Note that $p=[U_1]=A_{\eta}\cap G(3,W)$ is also invariant with respect to the involution.

\subsubsection{Fixed points computations}

Let $X_{A}$ be the IHS fourfold associated to the symmetric Verra fourfold $V_{A}$. Then $X_{A}$ admits three involutions: the covering involution $\mathbf{i}$ and two involutions induced by $\tilde{\iota}$ (which acts on $\wedge^3U_1\oplus( \wedge^2 U_1\otimes U_2$)). Let us call the latter $\mathbf{j}$ and $\mathbf{j}\circ \mathbf{i}$, where $\mathbf{j}$ is the symplectic involution among the two. 
\begin{prop}\label{fixed pts prop}
 The fixed locus of $\mathbf{j}$ in all cases consists of 28 points and a K3 surface isomorphic to the double cover of a del Pezzo surface $D_4$ of degree 4 branched in some curve $B\in |-2K_{D_4}|$.

Furthermore one of the following holds:
\begin{enumerate} 
\item $V_{A}\in \mathcal{V}_1$ and the fixed locus of  $\mathbf{i}\circ \mathbf{j}$ is the union of an abelian surface and a surface obtained as a double cover of a 12 nodal surface $S$ being a $(4,2)$ divisor in $\mathbb{P}^1\times \mathbb{P}^2$ branched in its nodes and in a curve  intersection of $S$ with a $(0,2)$ divisor in $\mathbb{P}^1\times \mathbb{P}^2$.
 
\item $V_{A}\in \mathcal{V}_2$  and the fixed locus of  $\mathbf{i}\circ \mathbf{j}$ is a surface $K$ obtained as a double cover of a 4 nodal del Pezzo surface $D$ of degree 4 branched over  its nodes and in a curve in the system $|-2K_D|$ i.e. K is an Enriques surface.
\end{enumerate}

\end{prop}
\begin{proof}
Let $Y_A$ be the EPW quartic associated to $A$. We shall first study the involution $\tilde{\iota}$ restricted to $Y_A$. 
Recall that $$Y_A=D^2_A\cap C_p\subset G(3,W)$$ where $C_p=G(3,W)\cap T_p$ 
with 
$T_p$ the projective tangent space to $G(3,W)$ at $p=[U_1]$. 
Let us denote by $H^+:=W^+\wedge \bigwedge^2 W^-$ and $H^-:=\wedge^3 W^- \oplus 
W^-\wedge \textstyle \bigwedge^2 V^+$ the eigenspaces of the action of $\kappa$ 
on $\wedge^3 W$. 
Let us now observe that the locus of fixed points of the action induced by $\kappa$ on 
$\mathbb{P}(\wedge^3 W)$ is 
$$ \mathbb{P}(H^+) \cup \mathbb{P} (H^-).$$ If we restrict to 
$G(3,W)$ we have three components of fixed points:
\begin{enumerate}
 \item $G_1:=\mathbb{P}(\wedge^3 W^-)\simeq \mathbb{P}^3$
 \item $G_2:=\mathbb{P}(W^-\wedge \textstyle \bigwedge^2 W^+)\simeq 
\mathbb{P}^3$
 \item $G_3:=\mathbb{P}^1\times G(2,4)\simeq \mathbb{P}(W^+)\times G(2,W^-) 
\subset \mathbb{P}(H^+).$
\end{enumerate}

Furthermore, for $i=1,2,3$ the restriction $\mathcal{T}_i$  of the projective tangent 
bundle $\mathcal{T}$ which is a subbundle of the trivial bundle $\wedge^3 W$ to $G_i$  
decomposes as a direct sum of
$$(\mathcal{T}_i \cap H^+)\oplus (\mathcal{T}_i \cap H^-)=:\mathcal{T}_i ^+\oplus \mathcal{T}_i^-$$ which 
gives respectively:
\begin{enumerate}
 \item $\mathcal{T}_1^-=4\mathcal{O}_{\mathbb{P}^3}$ and $\mathcal{T}_1^+=2( \wedge^2 
T_{\mathbb{P}^3} (-1))$
 \item $\mathcal{T}_2^-=4\mathcal{O}_{\mathbb{P}^3}$ and $\mathcal{T}_2^+=2( \wedge^2 
T_{\mathbb{P}^3} (-1) )$
 \item If we denote by $\pi_1$ and $\pi_2$ the projections of $\mathbb{P}(W^+)\times G(2,W^-)$ into the factors and 
by $\mathcal{U}_{G(2,W^-)}$ and 
 $\mathcal{Q}_{G(2,W^-)}$ the universal subbundle and quotient bundle then 
$$\mathcal{T}_3^-=\pi_2^* \mathcal{U}_{G(2,W^-)}\oplus \pi_2^* 
\mathcal{Q}^{\vee}_{G(2,W^-)}$$ and $\mathcal{T}_3^+$
 is the projective tangent bundle to
 $\mathbb{P}(W^+)\times G(2,W^-)$ in its Segre embedding i.e we have an exact 
sequence 
 $$0\to \pi_1^* \mathcal{O}_{\mathbb{P}(W^+)}(-1)\oplus \pi_2^* T_{G(2,W^-)}^{\vee} 
\to  (\mathcal{T}_3^+)^{\vee}\to \pi_1^*\mathcal{O}_{\mathbb{P}(W^+)}(1)\oplus 
\pi_1^*\mathcal{O}_{G(2,W^-)}(1).$$
\end{enumerate}

We now have three possibilities:
\begin{enumerate}
 \item $p\in G_1$ then $C_p\cap \mathbb{P}(H^-)=G_1=\mathbb{P}^3$ and $C_p\cap 
\mathbb{P}(H^+)=\mathbb{P}^1\times \mathbb{P}^2\subset G_3 $
 \item $p\in G_2$ then $C_p\cap \mathbb{P}(H^-)=G_2=\mathbb{P}^3$ and $C_p\cap 
\mathbb{P}(H^+)=\mathbb{P}^1\times \mathbb{P}^2 $
 \item $p\in G_3$ then $C_p\cap \mathbb{P}(H^-)=\mathbb{P}^1 \cup \mathbb{P}^1\subset G_1\cup G_2$,  
$C_p\cap \mathbb{P}(H^+)=c((\mathbb{P}^1\times\mathbb{P}^1)\cup \{pt\})
 $ 
\end{enumerate}
In each of these cases we determine the fixed locus of the involution $\tilde{\iota}$ on $Y_A=D^2_A\cap C_p$. For 
that note that $D^2_A\cap G_i=D^2_{A^+,\mathcal{T}_i^+}\cup D^2_{A^-,\mathcal{T}_i^-}\cup 
(D^1_{A^-,\mathcal{T}_i^-}
\cap D^1_{A^+,\mathcal{T}^+})$. We have the following possibilities:

\begin{enumerate}
 \item $p\in G_1$ then $G_1\subset D^1_{A^-,\mathcal{T}_1^-}$, whereas $ 
D^2_{A^-,\mathcal{T}_1^-}=\emptyset$ and the fixed locus on $Y_A$ is the union of a 
quartic hypersurface 
 $D^2_A \cap G_2=D^1_{A^+,\mathcal{T}_1^+}$ and two divisors $K_1=D^1_{A^-,\mathcal{T}_3^-}$ 
and $K_2=D^1_{A^+,\mathcal{T}_3^+}$ of bidegrees $(0,2)$ and $(4,2)$ in 
$\mathbb{P}^1\times \mathbb{P}^2 \subset G_3$ respectively.
 \item $p\in G_2$ then $G_1\subset D^1_{A^-,\mathcal{T}_1^-}$, whereas $ 
D^2_{A^-,\mathcal{T}_1^-}=\emptyset$ and the fixed locus on $Y_A$ is the union of a 
quartic hypersurface 
 $D^2_A \cap G_2=D^1_{A^+,\mathcal{T}_2^+}$ and two divisors $K_1=D^1_{A^-,\mathcal{T}_3^-}$ 
and $K_2=D^1_{A^+,\mathcal{T}_3^+}$ of bidegrees $(0,2)$ and $(4,2)$ in 
$\mathbb{P}^1\times \mathbb{P}^2 \subset G_3$ respectively.
 \item $p\in G_3$ then the fixed locus on $Y_A$ consists of 12 points (4 on each line) and two surfaces $J_1$ and $J_2$
 being quadric sections of $c(\mathbb{P}^1\times\mathbb{P}^1)$ corresponding to $D^1_{A^-,\mathcal{T}_3^-}$ and $D^2_{A^+,\mathcal{T}_3^+}$.
\end{enumerate}
We know that the involution $\tilde{\iota}$ on $Y_A$ lifts to both involutions $\mathbf{j}$ and $\mathbf{i}\circ\mathbf{j}$ on $X_A$. Note now that the preimage on $X_A$ of $D^3_A\cap C_p$ is contained in the fixed loci of both involutions $\mathbf{j}$ and $\mathbf{i}\circ\mathbf{j}$, furthermore non symplectic involutions
admit only smooth surfaces of fixed points whereas symplectic involutions admit isolated fixed points and smooth surfaces. 

Let us now determine $D^3_A\cap C_p$ in all three cases. 
For that observe that $$D^3_A\cap C_p\cap G_i=(D^2_{A^+,\mathcal{T}_i^+}\cap D^1_{A^-,\mathcal{T}_i^-}) \cup (D^1_{A^+,\mathcal{T}_i^+}\cap D^2_{A^-,\mathcal{T}_i^-}) \cup D^3_{A^-,\mathcal{T}_i^-} \cup D^3_{A^+,\mathcal{T}_i^+}$$
 recall that the class of a second Lagrangian degeneracy locus is computed by the formula:
$$D^2_{A',\mathcal{T}'}=c_1(\mathcal{T}')c_2(\mathcal{T}')-2c_3(\mathcal{T}').$$

We get 
$[D^2_{A^+,\mathcal{T}_1^+}]=16 [pt]$, $ [D^2_{A^+,\mathcal{T}_2^+}]=16 [pt]$, $[D^3_A\cap G_2]= [D^2_{A^+,\mathcal{T}_2^+}]=16 [pt]$, 
$[D^2_{A^-,\mathcal{T}_3^-}]=2h_1^3$, furthermore $[D^2_{A^+,\mathcal{T}_3^+}]=12h_1^2h_2$, where $h_1$ and $h_2$ are the pullbacks of the respective hyperplane sections of $G(2,4)$ and $\mathbb{P}^1$ to the product $G(2,4)\times \mathbb{P}^1$.
Since we know that the involution $\tilde{\iota}$ lifts to a symplectic involution $\mathbf{j}$ on $X$ the only possibilities for the fixed point locus of $\mathbf{j}$ on $X_A$ are:
\begin{enumerate}
\item $p\in G_1$  the fixed locus consists of preimages of 16 branch points on $G_1$ together with the preimage of 12 branch points (their class is computed as $12h_1^2h_2$ restricted to $\mathbb{P}^1\times \mathbb{P}^2\subset G(2,4)$) in $K_2$ and the K3 surface being the covering of $K_1$ branched in the intersection curve of $K_2\cap K_1$. 
\item $p\in G_2$  the fixed locus consists of preimages of 16 branch points on $G_2$ together with the preimage of 12 branch points (their class is computed as $12h_1^2h_2$ restricted to $\mathbb{P}^1\times \mathbb{P}^2\subset G(2,4)$) in $K_1$ and the K3 surface being the covering of $K_2$ branched in the intersection curve of $K_2\cap K_1$. 
 \item $p\in G_3$ the fixed locus consists of the preimages of the 12 isolated fixed points on $Y_A$ which are not in the branch locus, the preimage of 4 branch points in $J_1$ ($[D^2_{A^-,\mathcal{T}_3^-}]=2h_1^3$
 restricted to a special hyperplane section of one fiber) and the K3 surface being the double cover of $J_2$ branched in the curve of intersection with $J_1$. 
\end{enumerate}

Consider now the fixed points of the involution $\mathbf{i}\circ \mathbf{j}$. Note that the intersection of the fixed loci of $\mathbf{j}$ and $\mathbf{i}\circ \mathbf{j}$ is equal to the intersection of the fixed locus of $\mathbf{i}$ with the preimage of the fixed locus of involution on $Y_A$. Moreover the union of these loci is equal to the preimage of the fixed locus of $\tilde{\iota}$ on $Y_A$. We conclude that the fixed locus of $\mathbf{i}\circ \mathbf{j}$ is one of the following:
\begin{enumerate}
\item $p\in G_1$  the fixed locus consists of the preimage of the Kummer quartic  on $G_1$ which is an abelian surface together with the preimage of $K_2$. The latter preimage is a double cover of $K_2$ branched in the intersection $K_2\cap K_1$ and its 12 nodes.
\item$p\in G_2$  the fixed locus consists of the preimage of the Kummer quartic  on $G_2$ which is an abelian surface together with the preimage of $K_2$. The latter preimage is a double cover of $K_2$ branched in the intersection $K_2\cap K_1$. As earlier $K_2$ has 12 nodes and $K_2\cap K_1$ is smooth.
 \item $p\in G_3$ the fixed locus is  the preimage of $J_1$ i.e. a double cover of $J_1$ branched in the curve of intersection with $J_2$ and its 4 nodes. Note that such a double cover is an Enriques surface.
\end{enumerate}
\end{proof}

\begin{rem} The surfaces $K$ appearing as fixed loci of the involution $\mathbf{i}\circ \mathbf{j}$ in case 2 of Proposition \ref{fixed pts prop} form a $10$ dimensional family of polarized Enriques surfaces as $V_A$ moves in an $11$ dimensional family. Indeed, it is enough to prove that the branch curves $C_K$ of the cover $K\to D$ form a $10$ dimensional family of curves. For that we observe that by general lattice theory  the polarized $K3$ surfaces obtained as fixed loci of the symplectic involution $\mathbf{j}$ are in one to one correspondence with the family of IHS manifolds with involution given by varying $V_A$. Note now that the Del Pezzo surface $D$  is determined uniquely by the web of quadrics containing $C_K$, as the intersection of two rank 3 quadrics,  whereas the K3 surface is given by a pencil of quadrics meeting the pencil defining $D$ in a point corresponding to a singular quadric. It follows that a general curve $C_K$ lies on at most a one dimensional family of K3 surfaces in the family. 
This concludes the proof by dimension count. \end{rem}

\subsection{Picard lattices of moduli spaces of twisted sheaves}
For a more precise description of the Picard lattice of the families of IHS constructed in \S \ref{Verras}  we use Theorem \ref{moduli-twisted}. Both families constructed in \S \ref{Verras} are related to the family of K3 surfaces $S$ obtained as double covers of $\mathbb{P}^2$ branched along smooth symmetric sextics.
We present below some necessary ingredients.

We consider the family of K3 surfaces $S$ obtained as double covers of $\mathbb{P}^2$ branched along smooth sextics which are symmetric with respect to the involution $j:\PP^2\rightarrow \PP^2$. Let $H$ be a degree two polarization on $S$ which is invariant for the involution $\iota$ induced by $j$ on $S$.
The general member $S$ of this family is endowed with a non symplectic involution $\iota$ whose invariant lattice is $\Pic(S)\cong \langle 2\rangle\oplus E_8(-2)$ and whose transcendental lattice is $T(S)=U^2\oplus E_8(-2)\oplus \langle -2 \rangle$ (see \cite[Prop 2.2]{vGS} for further details).

Our aim is to find twists $\alpha \in \mathrm{Br}(S)$ and a Mukai vector $v\in H^{\ast}(S,\ZZ)$ such that the moduli space $M_v(S,\alpha)$  is an IHS fourfold of $K3^{[2]}$ type that admits a non symplectic involution with invariant lattice isomorphic either to $U(2)\oplus E_8(-2)$ or to $U\oplus E_8(-2)$.

Since an element $\alpha$ of order two defines a surjective homomorphism
$T(S)\to \ZZ/2\ZZ$,
our problem is reduced to find a homomorphism $\alpha\colon T(S)\to \Z/2\Z$ such that $\ker(\alpha)^{\perp}$ contains the lattice $U(2)\oplus E_8(-2)$.

Let $r_i, \,i=1,\dots, 9$, be a basis for $E_8(-2)\oplus \langle -2\rangle$. Then it is easy to show that $\alpha\in 
\mathrm{Br}(S)_2$ can always be written in the form 
\begin{equation}\label{alpha-form}
 \sum n_i r_i+\lambda\mapsto \frac{1}{2}\langle\sum\alpha_ir_i,\sum n_ir_i\rangle +\langle \lambda,\lambda_\alpha\rangle\,\mathrm{mod}\,(2)
\end{equation}
with $\lambda\in U^2$, $\lambda_\alpha\in U^2/2(U^2)$ and $\alpha_i=0,1$ for all $i=1,\dots,9.$ Indeed, for unimodular lattices all maps to $\mathbb{Z}/2\mathbb{Z}$ are obtained as reduction modulo $2$ of the pairing with an element of the lattice and our lattice can be written as a unimodular lattice direct sum twice a unimodular lattice, hence our claim.
In the following we will consider only two torsion classes having a $B$-lift $B\in H^2(S,\mathbb{Q})$ of $\alpha\in H^2(\oo_S^{\ast})$ such that $B.H=0$ and $B^2=0$. It was proven in \cite[\S 9]{vG} that these classes are precisely the classes corresponding to Verra fourfolds on a very general $K3$ surface of degree $2$ (cf.~\cite[Lem.~6.1]{Ku2}). As the conditions above are closed, this actually holds for all $K3$'s of degree two.

\begin{prop}\label{lem:15brauer-classes}
Let $S$ be as above and such that $\Pic(S)\cong \langle 2\rangle\oplus E_8(-2)$, and take $v:=(0,H,0)\in H^*(S,\Z)$. Then:
\begin{enumerate}
 \item the moduli space of stable sheaves $M_v(S)$ on $S$ has Picard lattice isometric to $U\oplus E_8(-2)$;
\item Let $\alpha$ be a non trivial Brauer class with the $B$-lift property outlined above i.e.~having a $B$-lift with $B^2=BH=0$. Then $\Pic(M_v(S,\alpha))\simeq U(2)\oplus E_8(-2)$.
\end{enumerate}
\end{prop}
\begin{proof}
First of all, $v=v_B$ (see Subsection \ref{ssec:twist}) for our choice of the $B$-lift of a Brauer class; also, $v^2=2$.
By the assumption, we have $T(S)\cong U^2\oplus E_8(-2)\oplus \langle -2\rangle$. For the trivial Brauer class, we thus have $$\Pic(M_v(S))=v^{\perp}\cap (U\oplus \Pic(S)).$$ There are two possible non-isometric choices for such a lattice, either it is $\langle 2\rangle\oplus E_8(-2)\oplus \langle -2\rangle$ or it is $U\oplus E_8(-2)$. This one occurs any time that $v^2=2$ and $\mathrm{div} (v)=2$, so in particular for $v=(0,H,0)$. This proves (1).

We denote with $\Gamma_\alpha$ the index two sublattice of $T(S)$ given by 
$\ker(\alpha)$ and by $B$ a B-lift of $\alpha$ in $(T(S)\oplus \Pic(S))\otimes \frac{1}{2}\mathbb{Z}$. Let $s,t$ be respectively the projection of $2B$ in the unimodular part of $T(S)$, $U^2$, and in its orthogonal $E_8(-2)\oplus\langle -2\rangle$. In order for $\lambda_\alpha$ to be non zero, we must have $s\neq 0$.
As before, $\Pic(S,\alpha)$ is the saturation of the sublattice generated by $\Pic(S)$ and $\langle (0,0,1),(2,s+t,0) \rangle$. As $t$ lies in the non unimodular part of $T(S)$ and the extension between this part and $\Pic(S)$ is unimodular and of maximal index, for all $t\neq 0$ there exists a non-zero $r\in \Pic(S)$ such that $(r+t)/2\in H^2(S,\mathbb{Z})$. Therefore, the $B$ lift $(s+t)/2$ corresponds to the same Brauer class of the lift $(r+s+t)/2=s/2+(r+t)/2$, which in turn is the same Brauer class of the lift $s/2$, as $(r+t)/2$ is an integer class. This means that if $s=0$ the Brauer class is in fact trivial. If $s\neq 0$, $\Pic(S,\alpha)=\langle \Pic(S),(0,0,1),(2,s,0) \rangle \cong U(2)\oplus E_8(-2)\oplus \langle 2\rangle$. Therefore, $v^\perp$ is isometric to the direct sum of $U(2)$ and $E_8(-2)=v^\perp\cap \Pic(S)$, which is point $(2)$.
\end{proof}
We will denote by $I$ the set of Brauer classes of case (2) above. Remark that all Brauer classes of order two are invariant with respect to $\iota$: the involution acts on $\mathrm{Br}(S)\cong T(S)^*\otimes \Q/\ZZ$ (see \cite{vG}) via $\alpha\mapsto \alpha\circ \iota^*$; on the other hand, $\iota^*_{|T(S)}=-\mathrm{id}_{T(S)}$, hence invariance is obvious for elements in $ \mathrm{Br}(S)_2$.
\begin{prop}\label{prop:e8_brau}
Let $S,\alpha,v$ be as above, let $\alpha\in I$ and let $S$ be such that $\Pic(S)\cong \langle 2\rangle\oplus E_8(-2)$, 
then $M_v(S,\alpha)$ is smooth and has a non symplectic involution with 
invariant lattice $U(2)\oplus E_8(-2)=\Pic(M_v(S,\eta))$. 
\end{prop}
\begin{proof} 
After choosing a general polarization, we can freely assume that 
$M_v(S,\alpha)$ 
is smooth (cf. \cite[prop.~3.10]{Yo}). By Lemma \ref{lem:15brauer-classes}, $\Pic(M_v(S,\alpha))=U(2)\oplus E_8(-2)$. As $v$ 
and $\alpha$ are $\iota$ invariant, there exists an induced involution on 
$M_v(S,\alpha)$ as explained in Section \ref{induced}. As the lattice $U(2)$ comes from $H^0\oplus H^4$ and $E_8(-2)$ 
is $\iota$ invariant, the induced involution leaves $\Pic(M_v(S,\alpha))$ 
invariant.
\end{proof}

\subsection{Proof of Theorem \ref{main}: the case $U(2)\oplus E_8(-2)$}\label{step}
 
Note that it follows from Section \ref{induced} that the considered IHS fourfolds are moduli spaces of twisted sheaves on a $K3$ surface and the non-symplectic involution is induced. Our aim is to describe 
the base $K3$ surface and the corresponding Brauer element.
 A general Verra fourfold $V$ from the family $\mathcal{V}_1$ (resp.~$\mathcal{V}_2$) gives rise through the Hilbert scheme of $(1,1)$ conics to an IHS fourfold $X_{V}$ with Picard group $$E_8(-2)\oplus U \ \ (\text{resp.}\ \ E_8(-2)\oplus U(2).)$$

Let $V\in \mathcal{V}_1$.
It follows from Proposition \ref{prop:e8_lagr} that $\Pic(X_V)$ is an over-lattice of $U(2)\oplus E_8(-2)$ which contains primitively $U(2)$ and $E_8(-2)$.
By Lemma \ref{discr}, one of the projections has singular discriminant locus and the preimage of the singular point is a corank 2 quadric. It follows that the Verra fourfold contains a reducible quadric. Each of these two planes is then a holomorphic section of the second projection with smooth discriminant with associated K3 surface $S_1$ of degree $2$.
We deduce by \cite[Thm.~6.3]{KPS} (see proof of \cite[Prop.~2]{ABBV}) that the Brauer element $\beta_1$ on $S_1$ corresponding
to the quadric bundle is trivial.
 By Theorem \ref{moduli-twisted} we have that $X_V$ is a moduli space of (untwisted) sheaves on $S_1$ with Mukai vector $(0,H,0)$.
 From Prop.~\ref{invV} we deduce that $S_1$ satisfies the assumption of Proposition \ref{lem:15brauer-classes}.
We can conclude that  $\Pic(X_V)$ is $U\oplus E_8(-2)$.

Now let $V\in \V_2$.
It follows from Proposition \ref{invV} that the general symmetric sextic curve can be the discriminant curve of a symmetric Verra threefold. 
One can prove using Macaulay 2 by studying general examples, that the two families of Verra threefolds $\V_1$ and $\V_2$ give rise to smooth IHS fourfolds (by finding the singular loci of the associated EPW quartics and using \cite[\S 3]{IKKR}). 
It follows from Theorem \ref{moduli-twisted} that $X_{V}$ is birational to the moduli space of twisted sheaves so it is isomorphic to it by Lemma \ref{chamber-orbits}. It follows from Lemma \ref{lem:15brauer-classes} that $X_{V}$
 has either Picard group $E_8(-2)\oplus U(2)$ or $E_8(-2)\oplus U$. Now our aim is to prove that the second case cannot happen.

We claim that the considered families of symmetric Verra threefolds $\V_1$ and $\V_2$ give
$11$ dimensional families of IHS fourfolds with non symplectic involutions.
Indeed, it is enough to prove that we obtain by our construction a general element either of the 
moduli space $\mathcal{M}_{10,10,0}^s$ or of $\mathcal{M}_{10,8,0}^s$.
Let $\mathcal{B}$ by the sub-variety of the Lagrangian Grassmannian parametrizing $A$ that corresponds to
the considered Verra fourfolds divided by $PGL(V)$ (as in Section \ref{fixed}).
Our family $\mathcal{B}$ induces a flat family $\mathcal{F}$ of polarized IHS fourfolds with $q=4$ such that 
\begin{enumerate}
\item the polarizations in each fiber $X_b$ of the family induce a 2:1 cover whose associated involution $\mathbf{i}_b$ has invariant lattice $F$ isomorphic to $U(2)$,
\item the family admits a flat family of symplectic involutions $\iota_b$ which commute with involutions $\mathbf{i}_b$ such that $\tau_b=\iota_b\circ \mathbf{i}_b$ form a flat family of non symplectic involutions with invariant lattice $F'$ isomorphic to a $2$-elementary overlattice of $U(2)\oplus E_8(-2)$ such that $F\subset F'$; either $F'\cong U\oplus E_8(-2)$ or $F'\cong U(2)\oplus E_8(-2)$. Let $(r,a,\delta)$ be the invariants associated to $F'$ in the sense of Section \ref{ssec:lattices}.
\end{enumerate}
Consider $\mathcal{B}^s\subset \mathcal{B}$ corresponding to those pairs $(X_b,\tau_b)\in\mathcal{M}_{r,a,\delta}^s$ which are simple pairs. Note that $\mathcal{B}^s$ is open by general theory and nonempty 
since we saw above that a very general element of $\mathcal{F}$ involves an IHS manifold $X$ with $\Pic(X)\simeq F'$ (from Proposition \ref{lem:15brauer-classes}). We hence have a map 
 \begin{equation}\label{kk}\psi: \mathcal{B}\ni A\to X_A\in \mathcal{M}_{r,a,\delta}^s.\end{equation}
Observe that $\psi$ has generically finite fibers. Indeed, consider a very general point in $b_0\in \mathcal{B}^s$, so that we can assume that $\Pic(X_{b_0})\cong F'$. Consider the set $$K=\{b\in\mathcal{B}^s |\ \psi(b)=\psi(b_0)\}.$$ We shall prove that $K$ is a finite set. Indeed $b\in K$ implies $X_b\simeq X_{b_0}$, furthermore on $X_{b_0}$ there is a countable set of polarizations of degree 4. Now, for each polarization of degree 4 inducing an involution on $X_{b}$ with fixed lattice $$U(2)\subset \Pic(X_b)=\Pic(X_{b_0})$$ there is at most one symplectic involution with anti-invariant lattice $$U(2)^{\perp}\subset \Pic(X)=\Pic(X_{b_0}).$$ It follows that for $b\in K$ the involution $\tau_b$ is uniquely determined by the choice of a polarization on $X_{b}$. We conclude that $K$ is countable and since it represents a fiber of an algebraic map it is finite.
We infer that the image of $\psi$ has dimension $11$ which is equal to the dimension of  $\mathcal{M}_{r,a,\delta}^s$ in both possible cases for $F'$.  It follows that $\psi$ is a dominant map which concludes the proof of the claim. 

On the other hand, by Corollary \ref{connectedness} the families $\mathcal{M}_{10,10,0}^s$ and $\mathcal{M}_{10,8,0}^s$ are irreducible.
So it is enough to show that a general element from $\mathcal{V}_1$ gives rise to IHS with involution that is different from the IHS's with involutions obtained from the family $\mathcal{V}_2$.
This can be done by comparing the fixed loci of the above involutions.
Indeed, by Proposition \ref{fixed pts prop} we infer that those fixed loci are Enriques surfaces in one case and a reducible smooth surface in the other case. 
We conclude that for a general Verra fourfold  $V \in\mathcal{V}_2$ we have $\Pic(X_{V})=E_8(-2)\oplus U(2)$. This completes the proof of the first statement in Theorem \ref{main}.\qed

\begin{rem}
When the invariant lattice is $U(2)\oplus E_8(-2)$, the \kahl cone of the manifold coincides with the positive cone, hence there is an infinite group of automorphisms. Indeed, by the proof of Morrison-Kawamata cone conjecture (see \cite{av} and \cite{my}), the round ample cone has a polyhedral fundamental chamber for the automorphism group action, hence the group is infinite. This infinite group is obtained by considering all different realization of the manifold as a double EPW quartic (which are given by all embeddings of $U(2)$ into $U(2)\oplus E_8(-2)$) and considering the (non commuting) associated involutions. 
\end{rem}

\section{Proof of Theorem \ref{main}: the case $U(2)\oplus D_4(-1)$}\label{sect:D4}
In this section we shall see that IHS fourfolds with Picard lattice $U(2)\oplus D_4(-1)$ can be constructed (as in Section \ref{EPW quartic}) from a singular Verra threefold having a singular point of type $x^2 + y^2 + z^3 + t^3 = 0$.
Note that by \cite{MW} the considered IHS fourfolds with Picard lattice $U(2)\oplus D_4(-1)$ can be constructed as moduli spaces of sheaves on K3 surfaces being double covers of $\PP^2$ branched along four nodal sextics.
Indeed, one can prove that $$U(2)\oplus D_4(-1)\oplus \langle 2\rangle \cong\langle -2\rangle^{\oplus 4}\oplus \langle 2\rangle \oplus U.$$

\subsection{Verra threefold with singular point $x^2 + y^2 + z^3 + t^3 = 0$ }
The proof of the second part of Theorem \ref{main} consists of the following two steps:
 \begin{prop}\label{asd}Let $A$ be a general Lagrangian subspace of $\wedge^3(U_1\oplus U_2)$ such that $\PP(A)\cap G(3,U_1\oplus U_2)= \{[N],[U_1]\}$ and $\dim \PP(T_N)\cap \PP(A)=2$. Then
$X_A$ has transversal singularities with the quadratic part of rank $1$ along a surface $P_N$ which is 
isomorphic to $\PP^1 \times \PP^1$.
 \end{prop} 
 \begin{cor} \label{sing}The singularities of $X_A$ are transversally $D_4$ along a $K3$ surface $R_N$ being the double cover of $P_N$ branched along a curve of bidegree $(4,4)$. There exists a resolution of singularities $\overline{X_A}\to X_A$ such that the exceptional locus is invariant with
 respect to the covering involution on $\overline{X_A}$ induced from the covering involution $X_A\to D_A$. Moreover, the invariant lattice of the latter involution is $U(2)\oplus D_4(-1)$.
  \end{cor}
 
 \begin{proof}[Proof of Proposition \ref{asd}]  The proof is similar to \cite[Lem.~2.3]{IKKR}.
First we claim that $P_N=C_{U_1}\cap C_{N}$ is isomorphic to $\PP^1\times \PP^1$ and  is contained in $D_A$.
Indeed, it is enough to remind that $C_U$ can be seen as parametrizing the spaces $U'$ that intersect $U$ in codimension 1.

 Let $U_0\in P_N$ be a generic point and choose a $3$-space $U_{\infty}$ such that $U_{\infty} \cap U_1 =0$, $U_{\infty} \cap U_0 =0$ and $\dim (U_{\infty}\cap N)=1$. Let
$$\mathcal{U} = \{[U ] \in G(3, U_1\oplus U_2 ) | \ U \cap U_{\infty} = 0\};$$
 then $ D_A \cap \mathcal{U}= \{[U]\in \mathcal{U} \cap C_{U_1} | \ \rk(\bar{q}_U -q_{\bar{A}})\leq 8 \}$ with the notation from the proof of  \cite[Lem.~2.3]{IKKR}.

Note that for each $[U]\in \mathcal{U}$ the quadric $q_U$ on $\mathbb{P}(T_{U_0})$ associated to the symmetric map $$Q_U: T_{U_0}\to T_{U_{\infty}}\simeq T_{U_0}^{\vee}$$ whose graph is $T_U$ contains the cone $C_{U_0}$. Furthermore, if 
$[U]\in C_{U_1}\cap  \mathcal{U}$ the quadric $q_U$ is moreover singular in $[U_1]\in \mathbb{P}(T_{U_0})$.
We thus have a natural map  $$\bar{\iota}|_{\mathcal{U}\cap C_{U_1}}:\mathcal{U}\cap C_{U_1}\to W_5=H^0(\mathcal{I}_{\hat{C}|\mathbb{P}^8}(2))$$ to the system of quadrics containing $\hat{C}\subset \PP^8$ the projection of $C_{U_0}$ from $[U_1]$. This map is a local isomorphism around $[U_0]$. Thus by restricting $\mathcal{U}$ to some $ [U_0]\in \mathcal{U}_0$
we can describe $\mathcal{U}_0 \cap D_A$ as isomorphic to a variety given by the determinant of a symmetric $9\times 9$ matrix $\bar{\iota}$ of power series in some open neighborhood of $0\in W_5$. Note that after the choice of $U_0$ we can choose $A$ general satisfying our assumptions such that additionally $$\PP(T_{[U_0]})\cap \mathbb{P}(A)=<[N],[U_1]>,$$ thus $\PP(\bar{T}_{[U_0]})\cap \mathbb{P}(\bar{A})=\{[N]\}$. By symmetric row and column operations  we may assume that the matrix of constant terms, which corresponds to $Q_{\bar{A}}$ is a diagonal matrix with one 0 and eight 1s on the diagonal. Let us now write the determinant in homogeneous parts: $$\Phi_0+\Phi_1+\Phi_2+\dots.$$
Observe that $\Phi_0=0$ since $$\mathrm{ker} \ \bar{Q}_{A}=<[N]>\neq 0$$ and $\Phi_1=0$ since $\bar{q}_{U}|_{<[N]>}=0$ for each $[U]\in T_{[U_0]}$ as $[N]\in C_{[U_0]}$. By the latter we also know that $\Phi_2$ is only dependent on the linear part of the matrix $\bar{\iota}$. The latter represents the system $W_5$ of quadrics containing $\hat{C}$.
Let $$K=<[N]>=\mathrm{ker}\ Q_{\bar{A}}=\bar{A}\cap T_{[U_0]} /\wedge^3 U_1.$$ Observe that the image of $K$ through the projection $C_{U_0}\to \hat{C}$ is a point that we denote by $\langle e \rangle$. Denote by $\PP(T)$ the projective tangent space to $\hat{C}\subset  T_{[U_0]} /\wedge^3 U_1$ at $\langle e \rangle$. We are hence in the context of \cite[Prop.~1.19]{O1}.

Our aim is to compute that $\rk (\Phi_2)=1$. 
Since $\dim T=6$ we infer that $\rk(\Phi_2)=3- \mathrm{cork}\ \bar{q}_{\bar{A}}|_{T/\langle e\rangle}$ where $ \bar{q}_{\bar{A}}$ is the quadric induced by $q_{\bar {A}}$ on $T/\langle e\rangle$.
It is enough to prove that $\mathrm{cork}\ \bar{q}_{\bar{A}}|_{T/\langle e\rangle}=2$ for $A$ satisfying the assumption. Note that since $Q_{\bar{A}}|_{<e>}=0$ we have $\mathrm{cork}\ \bar{q}_{\bar{A}}|_{T/\langle e\rangle}=\mathrm{cork} q_{\bar{A}}|_T -1$. We claim that 
$$\mathrm{rk} q_{\bar{A}}|_T= \dim (T_{[N]}\cap A).$$
Indeed, observe that $T=T_{U_0}\cap T_{[N]}$ and by our assumption $\dim(N\cap U_{\infty})=1$ we have $\dim (T_{[N]}\cap T_{[U_{\infty}]})=4$. It follows that $$T_{[N]}=T\oplus (T_{[N]}\cap T_{[U_{\infty}]});$$ now if $A\cap T_{[N]}=\Pi$ and $A$ is general such that $\dim \Pi=k$   then the image $\Psi$ of the projection  of $\Pi$ from $T_{[U_{\infty}]}$ is also a space of dimension $k$ contained in 
$T\subset T_{[U_{0}]}$. But $Q_A(\Pi)$ is the projection of $\Pi$ from $T_{[U_0]}$, it is hence contained in  $T_{[N]}\cap T_{[U_{\infty}]}= \mathrm{Ann(T)}$. It follows that the linear map $T\to T^{\vee}$ associated to $q_{\bar{A}}|_T$
has k-dimensional kernel which proves the claim and by putting $k=2$ our proposition.
 \end{proof}
 To prove Corollary \ref{sing} we shall prove that for every $[U_0]\in P_N=C_{U_1}\cap C_N$ there exists $A$ satisfying the assumptions of Proposition \ref{asd} for which the singularity of $D_A$ in some neighborhood $\mathcal{U}_0$ of $[U_0]$ is a transversal $D_4$ singularity along $R\cap \mathcal{U}_0$. Note that for a choice of general $[U_1], [N]$ all choices of $[U_0]$ are equivalent. Hence it is enough to prove our claim for one specific example. We do this using Macaulay 2. Indeed, we can choose points $\{x_0,\dots, x_5\}$ in $W$ such that \begin{center}$U_1=<x_0,x_1,x_2>$, $N=<x_0,x_3,x_4>$, $U_0=<x_0,x_1,x_3>$, $U_{\infty}=<x_3, x_4+x_2 ,x_5>$.
 \end{center} We then consider the subset of $C_{U_1}$ parametrizing graphs of maps $U_1\to U_{\infty}$ represented by $3\times 3$ matrices of the form $$ \left(  \begin{array}{ccc}
x&y(1+x)&z(1+x)\\
  t(1+x)&ty(1+x)&tz(1+x)\\
  u(1+x)&uy(1+x)&uz(1+x)
 \end{array} \right),$$
 with $|x|< 1$. It is a neighborhood of $[U_0]$ which here corresponds to the matrix 0. In these coordinates $\{x,y,z,t,u\}$ the matrix defining quadratic forms is 
 $$M_{Q_A}+\left( 
 \begin{array}{ccccccccc}  0&0&0& 0&x &t+xt& 0& y+yx&0\\
    0&0&0&-x&0 &u+xu& -y-yx&0&0\\
    0&0&0&-t-xt&-u-xu&0& 0& 0&0\\
    0&-x&-t-xt&0&0&0&   0&z+zx&0\\
    x&0&-u-ux&0&0&0&  -z-zx&0&0\\
    t+x&u+ux&0&0&0&0&   0&0&0\\
    0&-y-yx&0&0&-z-zx&0&0&0&0\\
    y+yx&0&0&z+zx&0&0&0&0&0\\
    0&0&0&0&0&0&0&0&0
 \end{array}
\right).$$
For a chosen random $A$ satisfying the assumptions we compute its determinant and write it in the form $F^2+G*H*(G+H)+ Y(F,G,H)$, where $F,G,H$ are linear forms and $Y$ is a polynomial of order 4 in $F,G,H$.
 
  It follows by Proposition \ref{asd} that for general $A$ satisfying our assumption the singularity is of transversal $D_4$ type along $R_N\cap \mathcal{U}_0$ in a neighborhood of $U_0$.  By compactness of $R_N$ we conclude that for a general $A$ satisfying our assumptions the singularity of $D_A$ is a transversal $D_4$ singularity along the whole $R_N$. 
  
  It follows that the lattice $U(2)\oplus D_4(-1)$ is contained in the invariant lattice of the covering involution on $X_A$. Since this lattice admits no finite extensions we can conclude by a dimension count. More precisely, it is enough to prove that the image of the constructed family of IHS fourfolds with involution to the moduli space $\mathcal{M}_{6,4,0}^s$ is of dimension 15. Notice 
  that each constructed $X_A$ admits an associated K3 surface $R_N$. We shall compute the dimension of the image of the associated family of K3 surfaces in the moduli space of polarized K3 surfaces. Recall that each $R_N$ is obtained as a double cover of $\mathbb{P}^1\times \mathbb{P}^1=C_{[U_1]}\cap C_{[N]}$ branched in a curve being a $(4,4)$ divisor 
  on $\mathbb{P}^1\times \mathbb{P}^1$ described as $D^3_A \cap P_N$. The latter curve is obtained as a first Lagrangian degeneracy locus on $P_N$ associated to the space $\bar{\bar{A}}=A/<[U_1],[N]>\subset <[U_1],[N]>^{\perp}/<[U_1],[N]>$ and the bundle defined as the quotient of the restriction to $P_N$ of the projective tangent bundle to $G(3,W)$   by the direct sum of two trivial subbundles given by $<[U_1]>$ and $<[N]>$. Fixing $U_2$ in such a way that $U_1  \cap U_2 =0$ and $\dim (N \cap U_2)=2$ we have:
  \begin{enumerate}
  \item for every $u\in P_N$ the fiber $T_u$ decomposes as \begin{equation}\label{Tu}T_u=(T_u\cap T_{U_1})\oplus (T_u\cap T_{U_2}) \end{equation} with $\PP(T_u\cap T_{U_1})$ being the projective tangent space to $P_N$ in $u$.
  \item $T_{N}$ decomposes as \begin{equation} \label{TN} T_N=(T_N\cap T_{U_1})\oplus (T_N\cap T_{U_2})
  \end{equation}  with $\dim (T_N\cap T_{U_1})=4$ and $\dim (T_N\cap T_{U_2})=6$.
  \end{enumerate}
Notice that under our assumption $A\cap T_{U_2}= <[N]> \neq 0$ so $A$ is not the graph of a symmetric map $T_{U_1}\to T_{U_2}=T_{U_1}^{\vee}$, however $\bar{\bar{A}}$ is a graph of a symmetric map $$\bar{\bar{q}}_A: \bar{\bar{T}}_{U_1}=T_{U_1}/<[U_1],[N]> \to (T_{U_2}\cap <[U_1],[N]>^{\perp})/<[U_1],[N]>=:\bar{\bar{T}}_{U_2}.$$
By our decomposition (\ref{TN}) of  $T_N$ we deduce that our assumption on the intersection of $A$ with $T_N$ translates to the assumption that the quadric $\bar{\bar{Q}}_A$ associated to $ \bar{\bar{q}}_A$ is of rank 1 when restricted to $(T_N\cap T_{U_1})/<[U_1],[N]>$. Furthermore, we claim that the $(4,4)$ divisor $D^3_A \cap P_N$ on $P_N$ is described as the degeneracy locus of the symmetric map $\tilde{q}: \mathcal{T}_{\mathbb{P}^2\times \mathbb{P}^2}|_{P_N}\to (\mathcal{T}_{\mathbb{P}^2\times \mathbb{P}^2}|_{P_N})^{\vee}$ induced by the restrictions of $\bar{\bar{Q}}_A$ to the projective tangent spaces to $\mathbb{P}^2\times \mathbb{P}^2$ forming a bundle that we denote $\mathcal{T}_{\mathbb{P}^2\times \mathbb{P}^2}$. Indeed, for $u\in C_{[U_1]}\cap C_{[N]}$ we have $T_u\cap A\supset <[U_1],[N]>$ hence $D^3_A \cap P_N=\{u\in C_{[U_1]}\cap C_{[N]}| \dim \bar{\bar {T}}_u\cap \bar{\bar{A}}\geq 1\}$. But we know that $p\in \bar{\bar {T}}_u\cap \bar{\bar{A}}$  if and only if  $p=\tilde{p}+\bar{\bar{q}}_A(\
tilde{p})$ with $\tilde{p}\in \bar{\bar{T}}_u\cap \bar{\bar{T}}_{U_1}$ and $\bar{\bar{q}}_A(\tilde{p})\in \bar{\bar{T}}_u \cap \bar{\bar{T}}_{U_2}$. The latter is equivalent to $\tilde{q}(\tilde{p})=0$ which proves our claim.
We continue by deducing that the cokernel of the map $\tilde{q}$ is a sheaf $\mathfrak{I}$ supported on the $(4,4)$ divisor $D^3_A \cap P_N$. Furthermore the restriction to this curve is a 2-quadratic line bundle (i.e. corresponding to a 2-torsion divisor twisted by $\mathcal{O}(2)$). 
We finally claim that the exact sequence:
$$0\to \mathcal{T}\xrightarrow{\tilde{q}} \mathcal{T}^{\vee}\to I\to 0$$
is a minimal locally free resolution of $\mathfrak{I}$ which is uniquely determined up to automorphism of $\mathcal{T}^{\vee}$. Indeed, from the exact sequence:
$$0\to \mathcal{O}_{P_N}(-1)\to \mathcal{T}\to  T_{\mathbb{P}^2\times \mathbb{P}^2}(-1)|_{P_N}\to 0$$
we easily check that $h^0(\mathcal{T})=h^1(\mathcal{T})=0$ and $h^0(\mathcal{T}^{\vee})=8$. It follows also that $\mathcal{T}^{\vee}$ is globally generated and that $h^0(\mathfrak{I})=8$ which leads to minimality of our resolution. 
Let us now take another locally free resolution:
$$0\to \mathcal{A}\rightarrow \mathcal{B}\to I\to 0.$$
which leads to a diagram
$$\begin{diagram}
0 &\rTo  &\mathcal{T} &\rTo^{\tilde{q}} & \mathcal{T}^{\vee}&\rTo & I&\rTo &0\\
  & &\uTo^{\psi} & &\uTo^{\phi}& &\uTo_{\mathrm{id}} & &\\
  0 &\rTo & \mathcal{A} &\rTo & \mathcal{B}&\rTo & I&\rTo &0\\
\end{diagram}
$$
then $\phi$ and $\psi$ must be surjective.
It follows that the resolution is unique up to automorphism of $\mathcal{T}^{\vee}$.

 We infer that $\bar{\bar{A}}$ is uniquely determined by $\mathfrak{I}$. Since a curve admits only finitely many 2-torsion divisors, for a fixed $(4,4)$ curve there are at most finitely many spaces $\bar{\bar{A}}$ and in consequence spaces $A$ defining this curve. We conclude that the codimension of the space of $(4,4)$ curves in our family is equal to the codimension of the family of Lagrangian spaces in our family. This finally means that the constructed family of IHS manifolds is of dimension 15.
 \subsubsection{Associated Verra 3-folds and discriminant curves}
 The aim of this section is to complete the proof of Theorem \ref{main}. This will be a direct consequence of
 Proposition \ref{D4}.
 Let $\overline{X_A}$ be an IHS fourfold with Picard group $U(2)\oplus D_4(-1)$ constructed in the previous section.
 \begin{prop}\label{D4}  The general IHS fourfold with Picard group $U(2)\oplus D_4(-1)$ can be constructed using the Verra fourfolds $V_A$
  with a singularity
of analytic type $x^2 + y^2 +z^2+ t^3 + u^3 = 0$ in $\C^5$. Moreover, the discriminant curves of such quadric bundles are sextics with triple points.
 \end{prop}
 \begin{proof} The idea is to study the Verra threefold related (see Section \ref{EPW quartic}) to the Lagrangian spaces considered above.  Let $A$ be as in Proposition \ref{asd}. Let $U$ be a vector space of dimension 6 with a chosen volume form and let $[U_1]\in G=G(3,U)$.  We use the notation from Section \ref{EPW quartic};
 denote moreover the projection of $[N]$ from $T_{[U_1]}$ onto $$T_{[U_2]}/<[U_2]>\simeq (T_{[U_1]}/<[U_1]>)^{\vee}$$ by $p$. Observe that if we see $N$ as the graph of a linear map $\phi_N: U_1 \to U_2$ 
 then $p=[\ker\phi_N\oplus \im \phi_N] $.
We want to prove that $V_A$ has singularity of local analytic type $y^2 + z^2 + t^3+u^3 = 0$ in the point $p$. For that we shall compute the tangent cone of $V'_A$ in $p$. 

Observe that this tangent cone is a hypersurface in the projective tangent space to $P_{U_2}$ in $p$ and is contained in the intersection of this projective tangent space  and the tangent cone in $p$ to $q_{A,U_2}$ for every $U_2$. Now if $q_{A,U_2}$ is singular in $p$ the considered tangent cone is exactly the intersection of  $q_{A,U_2}$ with the projective  tangent space to $P_{U_2}$ in $p$. To provide a  quadric  $q_{A,U_2}$ that is singular in $p$ let us consider a general $U_2$ such that $[U_2]\in \PP(T_{[N]})$. In that case $[N]\in C_{U_2}$ and its projection to $P_{U_2}$ is $p$, which implies that the quadric $q_{A,U_2}$ is singular in $p$.

Note that $(T_{[N]}\cap T_{[U_1]})$ is a vector space $\Pi^1_N$ of dimension 4. Also in that case $T_{[N]}\cap T_{[U_2]}$ is a 6 dimensional vector space $\Pi^2_{N}$. We have $T_{[N]}=\Pi^1_N\oplus \Pi^2_N$.

 Consider now the projection of $A\cap T_{[N]}$ from $T_{[U_1]}$ onto $T_{[U_2]}$. Let us denote its image by $D$. Clearly $D$ is a 3 space such that $D\subset \Pi^2_N$. But also 
$Q_{A,U_2}(D)$ is a 3 space in  $\Pi^1_N.$ Observe also that $$\Pi^1_N=(\Pi^2_N)^{\perp}\cap T_{[U_1]}/<[U_1]>.$$ 
It follows that $(Q^1_{A,U_2})|_{(\Pi^2_N)}$ defines a rank 2 quadric on $\Pi^2_N$. We finally observe that the projectivization of $\Pi^2_N$ is exactly the projective tangent space to $$P_{U_2} \subset \mathbb{P}(T_{[U_2]}/<[U_2]>)=\mathbb{P}((T_{[U_1]}/<[U_1]>)^{\vee})$$ in $p$. 

Summing up, our condition on $A$ is equivalent to the condition for the tangent cone in $p$ of the associated Verra threefold to be a rank 2 quadric of dimension 3. The latter is a codimension 4 condition in the space of quadrics. 
It follows by dimension count (note that the Lagrangian space is uniquely determined by a quadric containing the Verra fourfold)  that for the general $A$ in our family the Verra fourfold $V_A$ admits a single isolated singularity of analytic type equivalent to a hypersurface $x^2+y^2+z^2+t^3+u^3=0$ in $\mathbb{C}^5$.

 Let us finally determine the discriminant curves given by the projections from such a Verra threefold to both $\mathbb{P}^2$.
Consider one of the two projections. Clearly the projection of $[N]$ lies in the singular locus of the discriminant curve since $q_{\bar{A}}$ was rank 2 when restricted to the projective tangent in $[N]$. Furthermore, we can compute the quadratic term of the restriction of the equation of the discriminant to the affine chart for which the projection of $ [N]$ is $(0,0)$. We note that it is expressed in terms of the restriction of $q_{\bar{A}}$ to the 4-dimensional projective tangent in $[N]$. More precisely its 3 coefficients at the quadratic monomials on $\mathbb{C}^2$ are given by corresponding  3 minors of the matrix of the restriction of $q_{\bar{A}}$. It follows that in our affine chart the equation of the discriminant curve admits no quadratic terms and hence the curve admits a triple point in $(0,0)$ being the projection of $[N]$.
\end{proof}
\begin{rem} In order to prove that we obtain all IHS fourfolds with Picard lattice $D_4(-1)\oplus U(2)$ it is enough to observe that a general sextic with a triple point can be the discriminant of the considered singular Verra threefolds. Then we can generalize Theorem \ref{moduli-twisted}  in the case when the discriminant sextic is singular in order to show that $(1,1)$ conics on $V_A$ is birational to the moduli space of twisted sheaves on the K3 surface being the desingularization of the double cover of $\PP^2$ branched along the discriminant sextic with a triple point (this will be done in a future work).  We conclude as in Section \ref{step}.
\end{rem}
\begin{rem} We expect that we can also describe our moduli space geometrically using Verra fourfolds with discriminant being a four nodal sextic, however we will not prove it here.
Recall that four nodal Verra threefolds with four nodal discriminant sextic were studied before in \cite{FV} in order to study the geometry of the  universal abelian variety over $\mathcal{A}_5$.
They are however considering Verra threefolds having four nodes in different fibers of the other projection and
the moduli spaces of four nodal sextic curves $\Gamma \subset \PP^2$ with $\gamma \in \Pic^0(\Gamma)[2]$ inducing a non trivial double cover of the normalization. 
In our context it is natural to consider four nodal Verra threefolds inducing a split cover of the four nodal sextic such that the four nodes are on one fiber. 
\end{rem}
\begin{rem} Looking at pairs of discriminant sextics with triple points we can define a relation on the set of such sextic from the point of view of Dixon correspondence. Indeed our situation is analogous to the one studied in \cite{I} where the author studied discriminant sextics of one nodal Verra threefolds.
\end{rem}

\section{Unirationality - proof of Corollary \ref{cor}}\label{koniec}

In this section we conclude by showing the unirationality of the moduli spaces $\mathcal{M}_{10,10,0}^s$ and $\mathcal{M}_{6,4,0}^s$. As we saw before, the family with invariant lattice $D_4(-1)\oplus U(2)$ is induced \cite{MW}. Thus the unirationality of $\mathcal{M}_{6,4,0}^s$ can be treated from a more general point of view and follows from the following Proposition \ref{moduli1} and the fact that the space of four-nodal sextics is unirational.

Let $(X,\varphi)$ be a pair consisting of a smooth moduli space of stable sheaves of dimension $4$ on a K3 surface $S$ and a non symplectic involution $\varphi$ which is induced from a non symplectic involution $\psi$ on $S$.
Let $T\subset L$ be a primitive hyperbolic sublattice isometric to the fixed sublattice of $\varphi^*$ in $H^2(X,\ZZ)$; we call $(r,a,\delta)$ the invariants of this $2$-elementary sublattice. We use the notation introduced in Section \ref{moduli}.

Moreover, the pair $(S,\psi)$ above gives an element of the moduli space $\mathcal{M}_{r',a',\delta'}^{K3}$, which is constructed in \cite[Remark 4.5.3]{Nikulin_2el} and in \cite[Theorem 1.8]{Yoshikawa}, where $(r',a',\delta')$ are the invariants associated to the $2$-elementary sublattice $T'$ of the K3 lattice $ L_{K3}:=U^{\oplus 3}\oplus E_8(-1)^{\oplus 2}$ isometric to the invariant sublattice of  $\psi^*$ in $H^2(S,\ZZ)$.

The period map for $K3$ surfaces gives an isomorphism 
$\mathcal{M}_{r',a',\delta'}^{K3}\cong (D_{Z'}\setminus\mathcal{H}_{K3})/\Gamma_{Z'}$, where:
\begin{itemize}
\item $Z'$ is the orthogonal of $T'$ inside $L_{K3}$;
\item $\mathcal{H}_{K3}:=\cup_{\delta\in\Delta_{K3}(Z')}\delta^{\perp}$ with $\Delta_{K3}(Z'):=\lbrace \delta\in Z'|\delta^2=-2\rbrace$;
\item $\Gamma_{Z'}$ is the arithmetic subgroup of $O(Z')$ obtained as projection of the group $\Mon^2(L_{K3},T')$ of monodromies of $K3$ surfaces fixing $T'$.
\end{itemize} 

\begin{prop}\label{moduli1}
If $\mathcal{M}_{r',a',\delta'}^{K3}$ is unirational, the moduli space $\mathcal{M}_{r,a,\delta}^{s}$ is also unirational.
\end{prop}
\begin{proof}
The moduli spaces we are dealing with are quotients of the corresponding period domains by the groups of monodromy operators fixing the invariant lattice for the involutions. We will prove the proposition by showing that the period domains are isomorphic and the groups involved are one contained into the other, thus giving a surjective map $\mathcal{M}_{r',a',\delta'}^{K3}\dashrightarrow\mathcal{M}_{r,a,\delta}^{s}$.

By the Mukai isometry, the transcendental lattice of the $K3$ surface $S$ above is isometric with the transcendental lattice of any moduli space $X$ of stable sheaves on it, and this isometry is equivariant with the induced group action, as proven in \cite[Lemma 2.34]{MW}.  As a consequence, since for general elements in both moduli spaces $\mathcal{M}_{r',a',\delta'}^{K3}$ and $\mathcal{M}_{r,a,\delta}^{s}$ the transcendental lattice is isometric to the orthogonal to the invariant sublattice of the automorphism, we obtain that $Z\cong Z'$. As a consequence the period domain is in both cases $D_Z$ and $\Delta_{K3}(Z')\subset \Delta(Z)$.

In order to obtain the moduli space $\mathcal{M}_{r',a',\delta'}^{K3}$ of pairs $(S,\psi)$, we are removing the hyperplanes in $\mathcal{H}_{K3}$, whereas in order to construct $\mathcal{M}_{r,a,\delta}^{s}$ we have seen that we have to remove also the arrangement $\mathcal{H}_s$ of hyperplanes associated to $-10$ classes in $Z$ with divisibility $2$ and hyperplanes in $\mathcal{H}_K$. 

Let us look at the two groups involved: $\Gamma_{Z'}$ is the projection on $Z'\cong Z$ of the group of monodromy operators of $L_{K3}$ preserving the invariant lattice $T'$, and it is canonically isomorphic with the group of orientation preserving isometries of the Mukai lattice $\Lambda$ preserving $\Lambda^{1,1}$.
The group $\Gamma_{Z,K}$ is the projection on $Z'\cong Z$ of the group of monodromy operators of $L$ preserving the invariant lattice $T$ and a chamber $K$. We claim that $\Gamma_{Z,K}=\Gamma_Z$: indeed, given $g\in\Gamma_Z$, let $G\in\Mon^2(L,T)$ be such that $G_{|T}=g$. In particular, $G$ has spinor norm in $L\otimes \mathbb{R}$ equal to $1$, hence it leaves invariant the chosen connected component of the positive cone of $L\otimes \mathbb{R}$; this implies that $G$ preserves also $C_T$, and as a consequence the spinor norm of $G_{|T}=1$; this implies, by multiplicativity, that the spinor norm of $g$ in $Z\otimes \mathbb{R}$ is also $1$. The claim is then proven by observing that the extension $\tilde{G}$ of $\id_T\oplus g$ to $L$ has again spinor norm $1$, hence it is a monodromy operator which restricts to $g$ on $Z$ and which fixes $K$, i.e. $g\in\Gamma_{Z,K}$.

Given any element $(X,\eta)\in\mathcal{M}_T$ such that $X$ is a moduli space $M_v(S)$ of stable sheaves on a $K3$ surface and $\Pic(X)\cong T$, the elements of $\mathcal{G}:=\eta^{-1}\circ\Gamma_Z\circ \eta$ can be seen, as shown by Markman \cite{mark_tor}, as the group of orientation preserving isometries of the Mukai lattice preserving the Mukai vector $v$ and the invariant lattice $\Pic(X)$. Notice finally that $T(X)$ is the orthogonal complement of $\Pic(X)\oplus v$, therefore the primitive lattice containing $\Pic(X)$ and $v$ coincides with the algebraic part $\Lambda^{1,1}$ of the Mukai lattice $\Lambda$ and $\Gamma_Z$ coincides with the group of orientation preserving isometries preserving it. Thus, $\Gamma_Z=\Gamma_{Z'}$ and we get an isomorphism $\mathcal{M}_{r,a,\delta}^{s}\cong \mathcal{M}_{r',a',\delta'}^{K3}\setminus (\mathcal{H}_s/\Gamma_Z)$. Since $\mathcal{H}_s$ is locally finite for the action of $\Gamma_Z$, we obtain that the map extends to the desired rational map between the two moduli 
spaces.
\end{proof}
\begin{rem}
 The proof of Proposition \ref{moduli1} easily extends to moduli spaces of pairs $(X,\varphi)$ consisting of a moduli space of sheaves $X$ of dimension $2n$ on a K3 surface $S$ and an involution $\varphi$ such that $\varphi$ is induced from a non symplectic involution $\psi$ on $S$. The only difference is that in the case of manifolds of \kntiposp type for $n\geq 3$, the monodromy group of the Beauville--Bogomolov--Fujiki lattice $L_n:=U^{\oplus 3}\oplus E_8(-1)^{\oplus 2}\oplus \langle -2(n-1)\rangle$ is the subgroup of isometries which are orientation preserving and which induce $\pm \id$ on the discriminant group $L_n$; as a consequence, the arithmetic group $\Gamma_{Z'}$ is only contained inside $\Gamma_{Z}$.
 \end{rem}
 
 Another generalization is possible: to the period domains corresponding to $(\rho_{\varphi},T_{\varphi})$-polarized pairs $(X,\eta)$, with $X$ a moduli space of sheaves of dimension $2n$ on a K3 surface $S$ and $\eta$ a marking, when the $(\rho_{\varphi},T_{\varphi})$-polarization is defined, in the sense of \cite{BCS2}, starting from an induced non symplectic automorphism $\varphi$ of odd prime order. In particular, one could obtain the following:
\begin{prop}
Let $X$ be a moduli space of sheaves on a K3 surface $S$ and $\varphi\in\Aut(X)$ be a non symplectic automorphism of odd prime order which is induced from an automorphism $\psi\in\Aut(S)$. Suppose moreover that the period domain of $(\rho_{\psi},T_{\psi})$-polarized pairs $(S,\eta)$, in the sense of \cite{Dol-Kondo}, is unirational. Then the period domain for $(\rho_{\varphi},T_{\varphi})$-polarized pairs $(X,\eta')$ of \kntiposp is also unirational.
\end{prop}
We do not give a precise proof of this here in order to avoid introducing all the new notation required in the case of odd prime order, as introduced in \cite{Dol-Kondo} and \cite{BCS2}, since it is out of the scopes of this paper. Instead, we end this section with the proof of Corollary \ref{cor}.

\begin{proof}[Proof of Corollary \ref{cor}] 
In the case $U(2)\oplus E_8(-2)$, it is enough to observe that the space $\mathcal{V}_2$ of symmetric Verra
threefolds is unirational. We conclude by the considering the map (\ref{kk}) in the proof of Theorem \ref{main}.

In the case of $U(2)\oplus D_4(-1)$ the involution is induced by a non symplectic involution on a $K3$ surface with invariant lattice isometric to $\langle 2\rangle\oplus \langle -2\rangle^{\oplus 4}$, $2$-elementary sublattice with invariants $(5,5,1)$. By the work of Shepherd-Barron \cite[Theorem 6]{SB} and of Artebani and Kond\=o \cite[Theorem 2.7]{AK} (see also \cite{Ma}), the moduli space $\mathcal{M}_{5,5,1}^{K3}$ of $K3$ surfaces of this kind is rational. Hence, by Proposition \ref{moduli1} we deduce that $\mathcal{M}_{6,4,0}^s$ is unirational.\end{proof}


\begin{thebibliography}{PSSWA}
\bibitem[A]{A} Addington, N., \emph{On two rationality conjectures for cubic fourfolds}, Math. Res. Lett. 23 (2016), no. 1, 1--13.
\bibitem[AK]{AK} Artebani, M., Kond\=o, S., \emph{The moduli of curves of genus six and K3 surfaces}. Trans. Amer. Math. Soc. 363 (2011), 1445--1462.
\bibitem[AL]{AL} Addington, N., Lehn, M., \emph{On the symplectic eightfold associated to a Pfaffian cubic fourfold}, arXiv 1404.5657 appendix J. reine angew. Math. DOI 10.1515/crelle-2014-0145.
\bibitem[AV]{av}
Amerik, E., Verbitsky, M., \emph{Rational curves on hyper-K\"ahler manifolds}, Int. Math. Res. Not. IMRN 2015, no. 23, 13009--13045. 
\bibitem[AST]{AST}  Artebani, M., Sarti, A., Taki, S., \emph{K3 surfaces with non-symplectic automorphisms of prime order}, Math. Z. 268 (2011), no. 1-2, 507--533.
\bibitem[ABBV]{ABBV} Auel, A., Bernardara, M., Bolognesi, M., Varilly-Alvarado, A., \emph{Cubic fourfolds containing a plane and a quintic del Pezzo surface}, Algebraic Geometry 1 (2) (2014) 181--193.
\bibitem[BHPV]{BHPV} Barth, W.P., Hulek, K., Peters, C.A.M, and Van De Ven, A., \emph{Compact Complex Surfaces} 2nd Edition: A series of Modern Surveys in Mathematics, Springer-Verlag.
\bibitem[BHT]{bht}
 Bayer, A., Hassett, B., Tschinkel, Y.,\emph{Mori cones of holomorphic symplectic varieties of K3 type}, Ann. Sci. \`{E}c. Norm. Sup\'{e}r. (4) 48 (2015), no. 4, 941--950.
\bibitem[B1]{Beauville} Beauville, A., \emph{Vari\'et\'es K\"ahleriennes dont la premi\'ere classe de Chern est nulle},  J. Differential Geom. 18 (1983), no. 4, 755--782 (1984).
\bibitem[B2]{Beauville2} Beauville, A., \emph{Some remarks on K\"ahler manifolds with $c_1=0$. In: Classification of Algebraic and Analytic Manifolds,} Progr. Math., vol. 39, pp. 1-26. Birkh\:auser, Boston, MA (1983)
\bibitem[B3]{B} Beauville, A.,  \emph{Antisymplectic involutions of holomorphic symplectic manifolds}, Journal of Topology 4 no. 2 (2011), 300--304.
\bibitem[B4]{Beauville1} Beauville, A., \emph{Vanishing thetanulls on curves with involutions}, Rend. Circ. Mat. Palermo (2) 62 (2013), no. 1, 61--66. 
\bibitem[BCS]{BCS}
Boissi\`ere, S., Camere, C., Sarti, A., \emph{Classification of automorphisms on a deformation family of hyper-K\"ahler four-folds by p-elementary lattices}, Kyoto J. Math. 56 (2016), no. 3, 465--499.
\bibitem[BCS2]{BCS2} 
Boissi\`ere, S., Camere, C., Sarti, A., \emph{Complex ball quotients from manifolds of $K3^{[n]}$-type}, arXiv:1512.02067.
\bibitem[BCS3]{BCScoming}
Boissi\`ere, S., Camere, C., Sarti, A., in preparation.
\bibitem[BCMS]{BCMS} 
Boissi\`ere, S., Camere, C., Mongardi, G., Sarti, A., \emph{Isometries of ideal lattices and hyper-K\"ahler manifolds}, Int. Math. Res. Not. IMRN 2016, no. 4, 963--977.
\bibitem[C1]{Cam12}
 Camere, C., \em Symplectic involutions of holomorphic symplectic fourfolds, \em Bull. Lond. Math. Soc. vol. 44 (2012), no. 4 687--702.
\bibitem[C2]{Cam-AIF} Camere, C., \emph{Lattice polarized irreducible holomorphic symplectic manifolds}, Ann. Inst. Fourier, 66 no. 2 (2016), p. 687--709.
%
%
\bibitem[C]{C}  Catanese, F., \emph{Babbages conjecture, contact of surfaces, symmetric determinantal varieties and applications}, Invent. Math. 63 (1981), no. 3, 433--465.
\bibitem[CV]{CV}  Creutz, B., Viray, B., \emph{On Brauer groups of double covers of ruled surfaces},
Math. Ann. 362 (2015) pp. 1169--1200.
%
\bibitem[DM]{DM} Dolgachev, I., Markushevich,D., \emph{Lagrangian tens of planes, Enriques surfaces and
holomorphic symplectic fourfolds}, preprint, 2010. 
\bibitem[DK]{DK} Dolgachev, I., Kondo, S., \emph{The rationality of the moduli spaces of Coble surfaces and of nodal Enriques surfaces}, Izvestiya: Mathematics 77:3 509--524.
\bibitem[DK2]{Dol-Kondo} Dolgachev, I., Kondo, S., \emph{Moduli of K3 surfaces and complex ball quotients},  Arithmetic and geometry around hypergeometric functions, 43--100, Progr. Math., 260, Birkh\:auser, Basel, 2007.
%
\bibitem[FV]{FV} Farkas, G., Verra, A., \emph{The universal abelian variety over $A_5$}
 Annales Scientifiques de L'Ecole Normale Superieure 49(2016), 521--543.
\bibitem[Fe]{Ferretti} Ferretti, A., \emph{Special subvarieties of EPW sextics,} Math. Z. vol. 272 (2012), no. 3-4, 1137--1164.
\bibitem[GHS]{GHS}Gritsenko, V., Hulek, K., Sankaran, G. K.,
\emph{Moduli spaces of irreducible symplectic manifolds}, 
Compos. Math. 146 (2010), no. 2, 404--434.
\bibitem[H]{H} Huybrechts, D., \emph{The K3 Category of a cubic fourfold,} Compositio Mathematica 153 (2017), 586--620, arXiv 1505.01775.
\bibitem[H2]{H2} Huybrechts, D., \emph{The global Torelli theorem: classical, derived, twisted}. Algebraic geometry Seattle 2005. Part 1, 235--258, Proc. Sympos. Pure Math., 80, Part 1, Amer. Math. Soc., Providence, RI, 2009.
\bibitem[I]{I} Iliev, A., \emph{The Theta Divisor of Bidegree (2,2) Threefold in $\PP^2\times \PP^2$}, Pacific Journal of Mathematics
Vol. 180, No. 1, 1997.
\bibitem[IKKR]{IKKR} Iliev, A., Kapustka, G., Kapustka, M., Ranestad, K. \emph{Hyper-K\"ahler fourfolds and Kummer surfaces},	arXiv:1603.00403 [math.AG].
\bibitem[IKKR1]{IKKR1} Iliev, A., Kapustka, G., Kapustka, M., Ranestad, K. \emph{EPW cubes} to appear
in Journal f\"ur die reine und angewandte Mathematik DOI: 10.1515/crelle-2016-0044.
\bibitem[IOOV]{IOOV} Ingalls, C., Obus, A., Ozman, E., Viray, B., \emph{Unramified brauer classes on cyclic covers of the projective plane}. In Brauer groups and obstruction problems: moduli spaces and arithmetic, Progr. Math. Birkh\:auser Boston, Boston, MA, 2013.
%
\bibitem[Jou]{Joumaah}  Joumaah, M., \emph{Non-symplectic involutions of irreducible symplectic manifolds of $K3^{[n]}$-type}. Math. Z. 283 (2016), no. 3-4, 761--790.
\bibitem[KPS]{KPS}  Knus, M.A., Parimala, R., and Sridharan, R., \emph{On rank 4 quadratic spaces with given Arf and Witt invariants}, Math. Ann. 274 (1986), no. 2, 181--198.
\bibitem[Ku]{Kuz} Kuznetsov, Z., \emph{Derived categories of quadric fibrations and intersections of quadrics},
Adv. Math. 218 (2008), 1340--1369.
  \bibitem[Ku1]{Ku2} Kuznetsov, A., \emph{Derived categories of cubic fourfolds, in: Cohomological and geometric approaches to rationality
problems}, 219--243, Progr. Math. 282, Birkh\"auser Boston, Boston (2010).
%
\bibitem[KM]{KM}Kuznetsov A., and Markushevich, D., \emph{Symplectic structures on moduli spaces of sheaves via the Atiyah class}, J. Geom. Phys., 59(7):843--860, 2009. Also math/0703264.
\bibitem[Ma]{Ma} Ma, S., \emph{Rationality of the moduli spaces of 2-elementary K3 surfaces}, J. Algebraic Geom. 24 (2015), no. 1, 81--158.
%
\bibitem[Mar]{mark_tor}
 Markman, E., \emph{A survey of Torelli and monodromy results for \hk manifolds}, Proc. of the conference "Complex and Differential Geometry", Springer Proceedings in Mathematics (2011), Volume 8, 257--322.
\bibitem[MY]{my}
 Markman, E., and Yoshioka, K., \emph{A proof of the Kawamata-Morrison cone conjecture for holomorphic symplectic varieties of $K3^{[n]}$ or generalized Kummer deformation type,} Int. Math. Res. Not. IMRN 2015, no. 24, 13563--13574.
\bibitem[MS]{MS} Macr\`i, E., Stellari, P., \emph{Fano varieties of cubic fourfolds containing a plane.} Math. Ann. 354 (2012), no. 3, 1147--1176.
%
\bibitem[M]{m_invol}
Mongardi, G., \emph{Symplectic involutions on deformations of $K3^{[2]}$}, Centr. Eur. J. Math. vol. 10 (2012), no. 4 1472--1485.
\bibitem[M2]{m_wall}
Mongardi, G., \emph{A note on the \kahl and Mori cones of \hk manifolds}, Asian J. Math. 19 (2015), no. 4, 583--591.
\bibitem[MTW]{MTW} Mongardi, G., Tari, K., Wandel, M., \emph{Automorphisms of generalised Kummer fourfolds}, to appear in Manuscripta Math., arXiv 1512:00225.
\bibitem[MW]{MW} Mongardi, G., Wandel M., \emph{Induced automorphisms on irreducible symplectic manifolds.} J. Lond. Math. Soc. (2) 92 (2015), no. 1, 123--143.
\bibitem[Mo]{Mo} Morrison, D., \emph{On K3 surfaces with large Picard number} Invent, Math. 105--121(1984).
\bibitem[MYY]{myy}
Minamide, H., Yanagida, S., Yoshioka, K., \emph{Some moduli spaces of Bridgeland's stability conditions.} Int. Math. Res. Not. IMRN 2014, no. 19, 5264--5327. 
\bibitem[N1]{Nikulin_2el}  Nikulin, V.V, \emph{Factor groups of groups of the automorphisms of hyperbolic forms with respect
to subgroups generated by 2-reflections}. Soviet Math. Dokl., 20 (1979), 1156--1158.
\bibitem[N2]{Nikulin}Nikulin,  V.V., \emph{Integer symmetric bilinear forms and some of their geometric applications}, 
Izv. Akad. Nauk SSSR Ser. Mat. 43 (1979), no. 1, 111--177, 238.
\bibitem[N3]{N2} Nikulin, V.V., \emph{Discrete reflection groups in Lobachevsky spaces and algebraic surfaces}. Proceedings of the International Congress of Mathematicians, Vol. 1, 2 (Berkeley, Calif., 1986),
654--671, Amer. Math. Soc., Providence, RI, 1987.
\bibitem[OZ]{OZ}  Oguiso, K., Zhang, D.Q., \emph{K3 surfaces with order five automorphisms.} J. Math. Kyoto Univ. 38 (1998), no. 3, 419--438.
\bibitem[O1]{O1} O'Grady, K., \emph{Irreducible symplectic 4-folds and Eisenbud-Popescu-Walter sextics}, Duke Math. J. vol. 134 (2006), no. 1 99--137.
\bibitem[O2]{O2} O'Grady, K.G., \emph{Periods of double EPW-sextics}. Math. Z. 280 (2015), no. 1-2, 485--524.
\bibitem[PT]{PT} Perego, A. and Toma, M., \emph{Moduli spaces of bundles over non-projective K3 surfaces}. Kyoto J. Math. 57, no. 1 (2017), 107--146.
%
\bibitem[S]{Schreieder} Schreieder, S., \emph{Quadric surface bundles over surfaces and stable rationality}, arXiv:1706.01358.
%
%
\bibitem[SB]{SB} Shepherd-Barron, N. I., \emph{Invariant theory for $S_5$ and the rationality of $M_6$}. Compositio Math. 70 (1989), 13--25.
%
\bibitem[Ta]{Tari} Tari, K., \emph{Automorphismes des vari\'{e}t\'{e}s de
Kummer g\'{e}n\'{e}ralis\'{e}es}, PhD thesis, Universit\'{e} de Poitiers.
%
%
\bibitem[vGS]{vGS}  van Geemen, B., Sarti, A., \emph{Nikulin involutions on K3 surfaces}. Math. Z. 255 (2007), no. 4, 731--753.
\bibitem[vG]{vG}  van Geemen, B., \emph{Some remarks on Brauer groups of K3 surfaces}. Adv. Math. 197 (2005), no. 1, 222--247.
\bibitem[V1]{Vprym} Verra, A., \emph{The Prym map has degree two on plane sextics.} 
in Proceedings of the conference to commemorate the 50th anniversary of the death of Gino Fano. 735-759 (2004). MR2112601 
%
\bibitem[V2]{V} Verra, A., \emph{Geometry of genus 8 Nikulin surfaces and rationality of their moduli}, to appear in K3 Surfaces and their Moduli Proceedings of the in Schiermonnikoog 2014.
%
\bibitem[Y]{Yoshikawa}Yoshikawa, K.-I., \emph{K3 surfaces with involution, equivariant analytic torsion, and automorphic forms on the moduli space}. 
Invent. Math. 156 (2004), no. 1, 53--117.
%
%
\bibitem[Yo]{Yo} Yoshioka, K., \emph{Moduli spaces of twisted sheaves on a projective variety}, Moduli spaces and arithmetic geometry, 1--30, Adv. Stud. Pure Math., 45, Math. Soc. Japan, Tokyo, 2006.
\bibitem[Yo2]{yoshi_der} Yoshioka, K., \emph{Bridgeland's stability and the positive cone of the moduli spaces of stable objects on an abelian surface}, Development of moduli theory - Kyoto 2013, 473--537, Adv. Stud. Pure Math., 69, Math. Soc. Japan, [Tokyo], 2016.
\end{thebibliography}
\end{document}